\documentclass{amsart}
\usepackage{latexsym}
\usepackage{amssymb}
\usepackage{amsmath}
\usepackage{amscd}
\usepackage{graphicx}
\usepackage{float}
\usepackage{mathrsfs}
\usepackage{enumerate}
\usepackage{tikz}

\tikzset{node distance=1.5cm, auto}

\newtheorem{theorem}{Theorem}[section]
\newtheorem{proposition}[theorem]{Proposition}
\newtheorem{lemma}[theorem]{Lemma}
\newtheorem{corollary}[theorem]{Corollary}

\theoremstyle{definition}

\theoremstyle{remark}
\newtheorem{example}[theorem]{Example}
\newtheorem{remark}[theorem]{Remark}

\newcommand{\SL}{\operatorname{SL}}
\newcommand{\Vol}{\operatorname{Vol}}
\newcommand{\Area}{\operatorname{Area}}
\newcommand{\Len}{\operatorname{Length}}
\newcommand{\pr}{\operatorname{pr}}
\newcommand{\Log}{\operatorname{Log}}
\newcommand{\sgn}{\operatorname{sgn}}
\newcommand{\inter}{\operatorname{int}}
\newcommand{\Arg}{\operatorname{Arg}}
\newcommand{\Conv}{\operatorname{Conv}}
\newcommand{\ver}{\operatorname{Vert}}
\newcommand{\ord}{\operatorname{ord}}

\def\<{\langle}
\def\>{\rangle}

\def\RR{\mathbb{R}}

\def\CC{\mathbb{C}}
\def\ZZ{\mathbb{Z}}
\def\N{\mathcal{N}}
\def\a{\mathbf{a}}
\def\b{\mathbf{b}}
\def\u{\mathbf{u}}
\def\v{\mathbf{v}}
\def\e{\mathbf{e}}
\def\n{\mathbf{n}}
\def\0{\mathbf{0}}

\def\P{\mathcal{P}}
\def\cA{\mathcal{C}}
\def\Am{\mathcal{A}}
\def\H{\mathcal{H}}

\def\E{\mathcal{E}}
\def\I{\mathcal{I}}

\def\cL{\mathcal{L}}
\def\T{\mathbf{T}}
\def\Z{\mathcal{Z}}

\title{Coamoebas of polynomials supported on circuits}
\date{\today}
\author{Jens Forsg{\aa}rd}
\address{Department of Mathematics \\ Stockholm University \\
SE-106 91 Stockholm, Sweden.}
\email{jensf@math.su.se}

\begin{document}

\begin{abstract}
We study coamoebas
of polynomials supported on circuits.
Our results include
an explicit description of the space of coamoebas, 
a relation between connected components of the coamoeba complement
and critical points of the polynomial,
an upper bound on the area of a planar coamoeba,
and a recovered bound on the number of positive solutions
of a fewnomial system. 
\end{abstract}

\maketitle

\section{Introduction}
\label{sec:Intro}
A possibly degenerate circuit is a point configuration
$A\subset \ZZ^n$ of cardinality $n+2$ which span
a sublattice $\ZZ A$ of rank $n$. That is, such that
the Newton polytope $\N_A = \Conv(A)$ is
of full dimension. A polynomial system $f(z) = 0$
is said to be supported on a circuit $A$ if
each polynomial occurring in $f(z)$ is supported on $A$.
Polynomial systems supported on circuits have recently been
been studied in the context of, e.g., real algebraic
geometry \cite{Bi11,BS11}, complexity theory \cite{BRS09},
and amoeba theory \cite{TdW13}. 
The name ``circuit'' originate from matroid theory;
see \cite{S96} and \cite{Z95} for further background.

The aim of this article is to describe geometrical and topological
properties of coamoebas of polynomials
supported on circuits. Such an investigation is motivated
not only by the vast number of applications of circuits in different areas
of geometry, 
but also since circuits provide an ideal testing ground for
open problems in coamoeba theory.

This paper is organized as follows.
In Section~\ref{sec:coamoebas} we will give a brief overview of coamoeba theory.
In Section~\ref{sec:RealPoints} we will discuss the relation between real polynomials and 
the coamoeba of the $A$-discriminant.
The main results of this paper are contained in 
Sections~\ref{sec:configurationspace}--\ref{sec:systems},
each of which can be read as a standalone text.

In Section~\ref{sec:configurationspace} we will give a complete description
of the space of coamoebas. That is, we will describe how the topology of 
the coamoeba $\cA_f$ depends on the coefficients of $f$. Describing the space of amoebas
is the topic of the articles \cite{TdW13} and \cite{TdW14},
and to fully appreciate our result one should consider these spaces simultaneously,
see e.g.\ Figure~\ref{fig:configurationspace}.
The geometry of the space of coamoebas is closely related to
the $A$-discriminantal variety, see Theorems~\ref{thm:configurationspaceinteriorpoint}
and \ref{thm:configurationspacevertex}.

In Section~\ref{sec:maximalarea} we will prove that the 
area of a planar circuit coamoeba is bounded
from above by $2\pi^2$.
That is, a planar circuit coamoeba covers at most half of the torus $\T^2$. 
Furthermore, we will prove that a circuit admits a coamoeba of maximal area
if and only if it admits an equimodular triangulation.
Note that we calculate area without multiplicities, in contrast to \cite{Mik14}.
However, the relation between (co)amoebas of maximal area and
Harnack curves is made visible also in this setting.

In Section~\ref{sec:criticalpoints} we will prove that,
under certain assumptions on $A$, the critical points of $f(z)$
are projected by the componentwise argument mapping into distinct
connected component of the complement of the coamoeba $\cA_f$.
Furthermore, this projection gives
a bijective relation between the set of critical points 
and the set of connected components of the complement of the 
closed coamoeba.
This settles a conjecture used in in \cite{JvS06}
when computing monodromy in the context of dimer models 
and mirror symmetry.

In Section~\ref{sec:systems} we will consider bivariate systems 
supported on a circuit.
If such a system is real, then it admits at most three roots in
$\RR_+^2$. 
The main contribution of this section is that
we offer a new approach to fewnomial theory.
Using our method, we will prove that if $\N_A$ is a simplex,
then, for each $\theta\in\T^2$, a complex bivariate system 
supported on $A$ has at most
two roots in the \emph{sector} $\Arg^{-1}(\theta)$.

\section{Coamoebas and lopsidedness}
\label{sec:coamoebas}
Let $A$ denote a point configuration 
$A = \left\{\a_{0}, \dots, \a_{N-1}\right\}\subset \ZZ^n$,
where $N = \#A$. By abuse of notation, we identify $A$ with the 
$(1+n)\times N$-matrix 
\begin{equation}
\label{eqn:Amatrix}
A = \left(\begin{array}{ccc} 
 1 &   \dots & 1\\
\a_0 &  \dots & \a_{N-1}\end{array}\right).
\end{equation}
The \emph{codimension} of $A$ is the integer $m = N - 1 - n$.
A \emph{circuit} is a point configuration of codimension one.
A circuit is said to be nondegenerate if it is not a pyramid over a circuit of
smaller dimension. That is, if all maximal minors of the matrix $A$ are 
nonvanishing. We will partition the set of circuits into two classes;
\emph{simplex circuits}, for which $\N_A$ is a simplex, 
and \emph{vertex circuits}, for which $A = \ver(\N_A)$.

We associate to $A$ the family $\CC_*^{A}$ consisting of 
all polynomials
\[
f(z) = \sum_{k=0}^{N-1} f_{k}\,z^{\a_k},
\]
where $f(z)$ is identified with the point
$f = (f_0, \dots, f_{N-1}) \in \CC_*^{A}$.
By slight abuse of notation, we will denote by $f_k(z)$ the monomial
function $z\mapsto f_k\,z^{\a_k}$.
We denote the algebraic set defined by $f$ by $Z(f) \subset \CC_*^n$.
The \emph{coamoeba}  $\cA_f$ is the image of $Z(f)$ under 
the componentwise argument mapping $\Arg\colon \CC_*^n \rightarrow \T^n$
defined by
\[
\Arg(z) = (\arg(z_1), \dots, \arg(z_n)),
\]
where $\T^n$ denotes the real $n$-torus.
It is sometimes beneficial to consider the multivalued argument mapping,
which gives the coamoeba as a multiply periodic subset of $\RR^n$.
Coameobas were introduced by Passare and Tsikh as a dual object, in an imprecise sense, of
the amoeba $\Am_f$.

We will say that a point $z\in \CC_*^n$ is a \emph{critical point} of $f$ if it solves
the system
\begin{equation}
\label{eqn:criticalpoint}
\partial_1 f(z) = \dots = \partial_nf(z) = 0.
\end{equation}
If in addition $z\in Z(f)$ then $z$ will be called a \emph{singular point} 
of $f$.
The $A$-discriminant $\Delta(f) = \Delta_A(f)$ is an irreducible
polynomial with domain $\CC_*^A$ which vanishes
if and only if $f$ has a singular point in $\CC_*^n$ \cite{GKZ94}.

A \emph{Gale dual} of $A$ is an integer matrix 
$B$ whose columns span the right 
$\ZZ$-kernel of $A$.
That is, $B$ is an integer $N\times m$-matrix,
of full rank, such that its maximal minors are relatively prime.
A Gale dual is unique up to the action of $\SL_m(\ZZ)$. 
The rows $\b$ of $B$ are indexed by the points $\a_k\in A$.
To each Gale dual we associate a zonotope
\[
 \Z_B = \left\{\frac\pi 2 \sum_{k=0}^{N-1} \lambda_k \b_k \, \bigg| \, |\lambda_k| \leq 1 \right\} \subset \RR^m.
\]

We will say that a triangulation $T$ of $\N_A$ is a triangulation of $A$ if $\ver(T) \subset A$.
Such a triangulation is said to be \emph{equimodular} if all maximal simplices
has equal volume.

Let $h$ be a \emph{height function} $h\colon A \rightarrow \RR$. The function $h$ induces a
triangulation $T_h$ of $A$ in the following manner. Let $\N_h$ denote the
polytope in $\RR^{n+1}$ with vertices $(\a, h(\a))$. The lower facets of $\N_h$ are the facets
whose outward normal vector has negative last coordinate. Then, 
$T_h$ is the triangulation of $A$ whose maximal simplices are the images of the lower facets
of $\N_h$ under the projection onto the first $n$ coordinates. A triangulation $T$ of
$A$ is said to be \emph{coherent} if there exists a height function $h$ such that $T =T_h$.

If $A$ is a circuit then $B$ is a column vector, unique up to sign. 
Hence, the zonotope $\Z_B$ is an interval.
Let $A_k = A\setminus\{\a_k\}$, with associated matrix $A_k$, and let 
$V_k = \Vol(A_k)$. 
If $A$ is a nondegenerate circuit, 
so that $V_k > 0$ for all $k$,
then $\N_A$ admits exactly two coherent triangulations with 
vertices in $A$ \cite{GKZ94}. 
Denote these two triangulations by $T_\delta$ for
$\delta \in \{ \pm 1\}$. Each simplex $\N_{A_k}$ occurs in exactly one
of the triangulations $T_\delta$. Hence, there is a well-defined assignment 
of signs $k\mapsto \delta_k$, where
$\delta_k \in \{\pm 1\}$, such that 
\[
 T_\delta = \big\{\N_{A_k}\big\}_{\delta_k = \delta}, \quad \delta = \pm1.
\]
Here, we have identified a triangulation with its set of maximal simplices.
As shown in \cite[chp.\ 7 and chp.\ 9]{GKZ94} and \cite[sec.\ 5]{FJ15}, a Gale dual of $A$ is given by
\begin{equation}
\label{eqn:Bcircuit}
 \b_k = (-1)^k |A_k| = \delta_k V_k.
\end{equation}
Thus, the zonotope $\Z_B$ is an interval of length $2\pi \Vol(A)$.

The $A$-discriminant $\Delta$ has $n+1$ homogeneities, one for each row
of the matrix $A$. Each Gale dual correspond to a dehomogenization
of $\Delta$. To be specific, introducing the variables
\begin{equation}
\label{eqn:ReducedVariables}
 \xi_j = \prod_{k=0}^{N-1}f_{k}^{\b_{k j}},\quad j = 1, \dots, m,
\end{equation}
there is a Laurent monomial $M(c)$ and a polynomial $\Delta_B(\xi)$ such that
\[
 \Delta_B(\xi) = M(f)\Delta_A(f).
\]
We will say that $\Delta_B$ is the \emph{reduced form} of $\Delta$.
Such a reduction yields a projection $\pr_B\colon \CC_*^A \rightarrow \CC_*^m$,
and we will say that $\CC_*^m$ is the reduced family associated to $A$, and
that $\pr_B(f)$ is the reduced form of $f$.
\begin{example}
Let $A = \{0,1,2\}$, so that $\CC_*^A$ is the family of quadratic univariate polynomials
\[
f(z) = f_0 + f_1z + f_2 z^2.
\]
Consider the Gale dual $B = (1,-2,1)^t$, and introduce the variable $\xi = f_0f_1^{-2}f_2$.
In this case the $A$-discriminant $\Delta_A$ is well-known, and we find that
\[
f_1^{-2}\Delta_A(f) = f_1^{-2}\left(f_1^2 - 4f_0 f_2\right) = 1 - 4 \xi = \Delta_B(\xi).
\]
The projection $\pr_B$ correspond to performing the change of variables $z\mapsto f_0f_1^{-1} z$,
and multiplying $f(z)$ by $f_0^{-1}$, after with we obtain the reduced family
consisting of all polynomials of the form
\[
f(z) = 1 + z + \xi z^2.
\]
\end{example}

Let $S$ denote a subset of $A$. The \emph{truncated polynomial}
$f_S$ is the image of $f$ under the projection 
$\pr_S\colon \CC_*^A\rightarrow \CC_*^{S}$.
Of particular interest is the case when
$S =\Gamma\cap A$ for some face $\Gamma$ of the Newton polytope $\N_A$
(denoted by $\Gamma\prec\N_A$).
We will write $f_\Gamma = f_{\Gamma\cap A}$. 
It was shown in \cite{NS13} that
\[
\overline{\cA}_f = \bigcup_{\Gamma\prec\N_A} \cA_{f_\Gamma},
\]
Let $\E$ denote the set of edges of $\N_A$, then the \emph{shell} of the coamoeba is defined by
\[
\H_f = \bigcup_{\Gamma\in\E} \cA_{f_\Gamma}.
\]
As an edge $\Gamma$ is one dimensional, the shell $\H_f$
is a hyperplane arrangement. Its importance can 
be seen in that each full-dimensional cell of $\H_f$ contain at most one 
connected component of the complement of $\overline{\cA}_f$, see \cite{FJ15}.

\begin{example}
The coamoeba of 
$f(z) = 1 + z_1 + z_2$,
as described in \cite{FJ15} and \cite{NS13},
can be seen in Figure~\ref{fig:simplexcoamoeba},
where it is drawn in the fundamental domains 
$[-\pi,\pi]^2$ and $[0,2\pi]^2$.
The shell $\H_f$ consist of the hyperplane arrangement
drawn in black. In this case, it is equal to the boundary of $\cA_f$.
The Newton polytope $\N_A$
and  its outward normal vectors
are drawn in the rightmost picture.
If $\H_f$ is given 
orientations in accordance with the outward normal
vectors of $\N_A$, then the interior of the coamoeba consist 
of the oriented cells.
\begin{figure}[h]
\centering
\includegraphics[width=30mm]{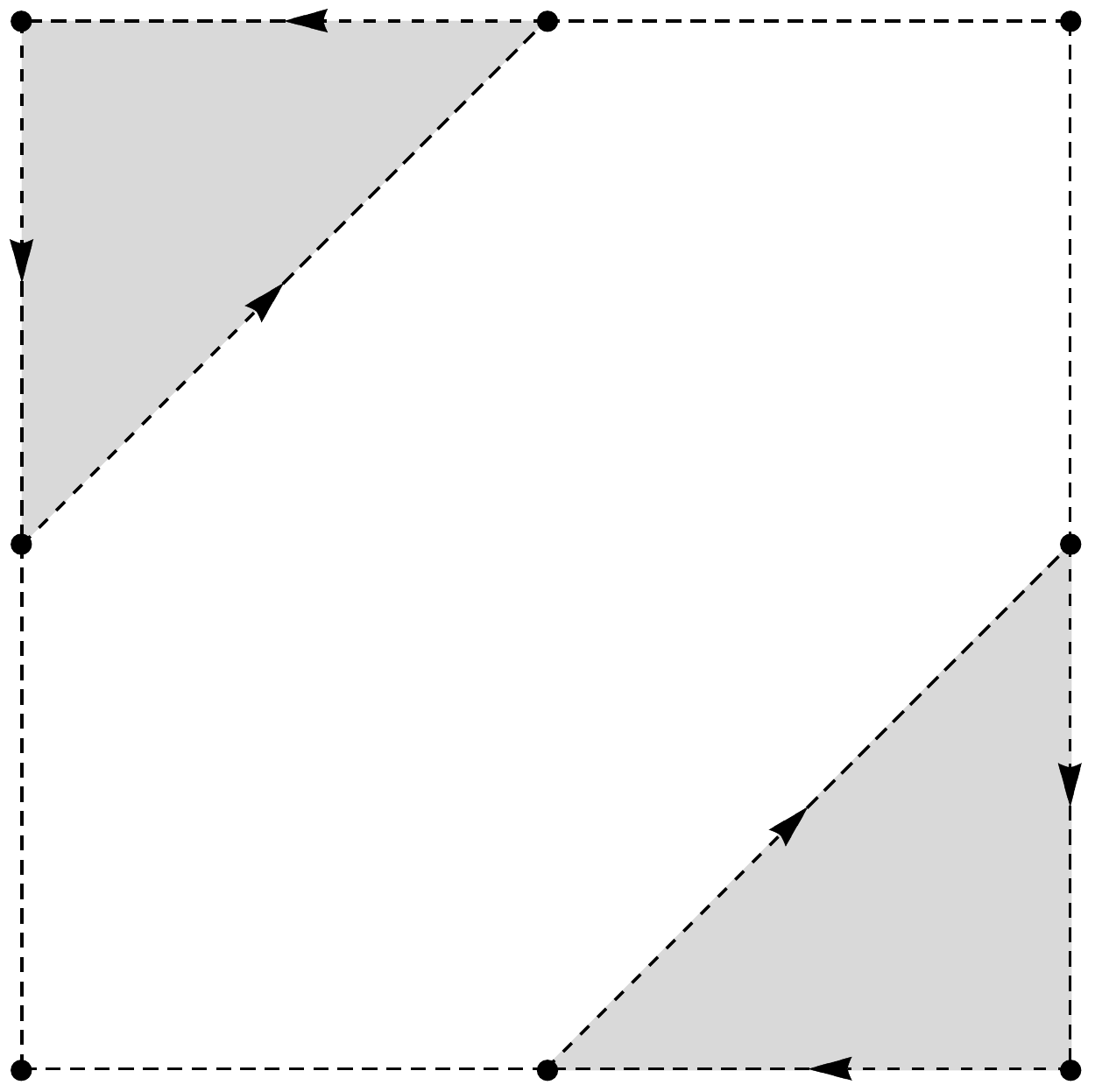}
\hspace{10mm}
\includegraphics[width=30mm]{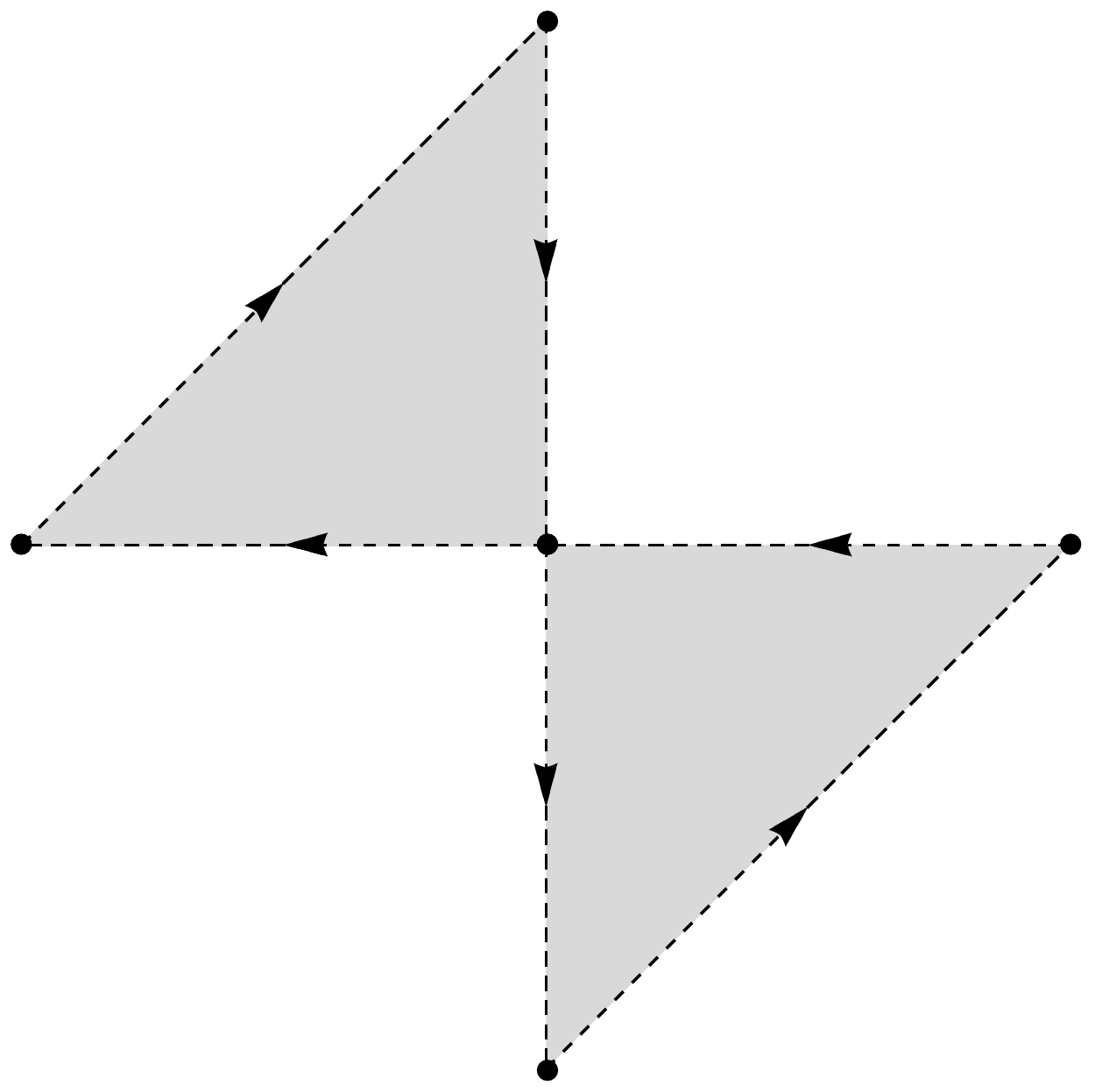}
\hspace{10mm}
\includegraphics[width=30mm]{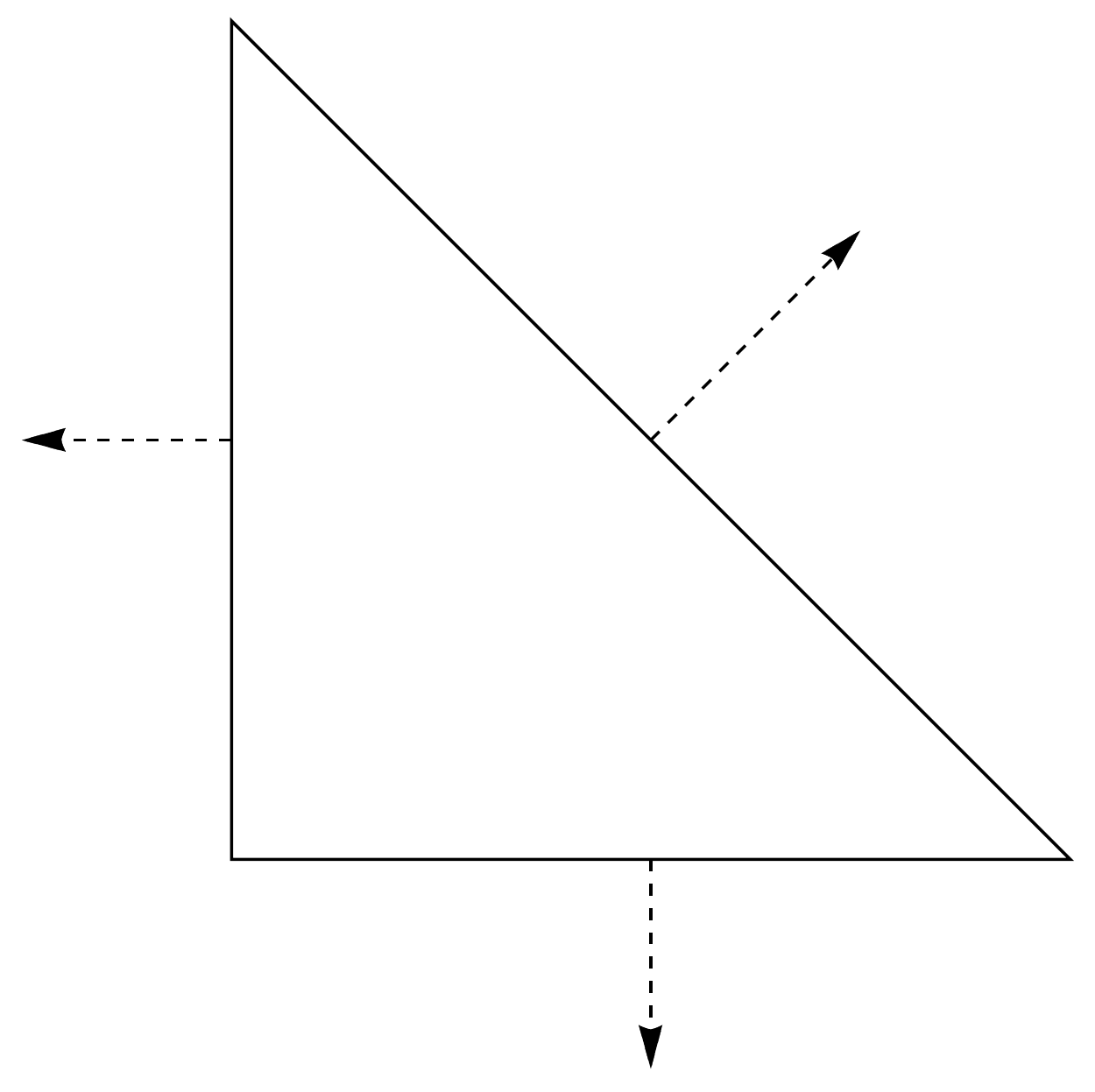}
\caption{The coamoeba of $f(z) = 1 + z_1 + z_2$ in two fundamental regions, and
the Newton polytope $\N_A$.}
\label{fig:simplexcoamoeba}
\end{figure}
\end{example}

Acting on $A$ by an \emph{integer affine transformation}
is equivalent to performing a monomial change of variables and
multiplying $f$ by a Laurent monomial. Such an action induces 
a linear transformation of the coamoeba $\cA_f$, when viewed 
in $\RR^n$ \cite{FJ15}. 
We will repeatedly use this fact to impose assumptions on $A$, e.g., 
that it contains the origin.

The polynomial $f$ is said to be \emph{colopsided} 
at a point $\theta\in \RR^n$ 
if there exist a phase $\varphi$ such that
\begin{equation}
\label{eqn:colopsidedness}
 \Re\left(e^{i\varphi}f_k(e^{i \theta})\right) \geq 0, \quad k = 0, \dots, N-1,
\end{equation}
with at least one of the inequalities \eqref{eqn:colopsidedness}
being strict. 
The motivation for this definition is as follows.
If $f$
is colopsided at $\theta$, then
\[
 \Re\left(e^{i\varphi}f(re^{i\theta})\right) = 
 \sum_{k=0}^{N-1}r^{\a_k}\,\Re\left(e^{i\varphi}f_k(e^{i\theta})\right) > 0, \quad \forall\,r\in \RR_+^n,
\]
since at least one term of the sum is strictly positive. 
Hence, colopsidedness at $\theta$ implies that $\theta \in\T^n\setminus\cA_f$. 
The colopsided coamoeba, denoted $\cL_f$, is defined as the set of all $\theta$ such that 
\eqref{eqn:colopsidedness} does not hold for any phase $\varphi$ \cite{FJ15}.
Hence, $\cA_f \subset \cL_f$. 

Each monomial $f_k(z)$ defines an affinity 
(i.e., a group homomorphism composed with a translation)
$f_k\colon \CC_*^n \rightarrow \CC_*$ by $z\mapsto f_k\,z^{\a_k}$.
We thus obtain unique affinities $|f_k|$ and $\hat f_k$ such that the
following diagram of short exact sequences 
commutes:
\begin{center}
\begin{tikzpicture}
  \node (O1) {0};
  \node (R1) [right of=O1] {$\RR_+^n$};
  \node (C1) [right of=R1] {$\CC_*^n$};
  \node (T1) [right of=C1] {$(S^1)^n$};
  \node (E1) [right of=T1] {$0$};
  \node (O2) [below of=O1] {$0$};
  \node (R2) [right of=O2] {$\RR_+$};
  \node (C2) [right of=R2] {$\CC_*$};
  \node (T2) [right of=C2] {$S^1$};
  \node (E2) [right of=T2] {$0$.};
  \draw[->] (O1) to node {} (R1);
  \draw[->] (R1) to node {} (C1);
  \draw[->] (C1) to node {} (T1);
  \draw[->] (T1) to node {} (E1);
  \draw[->] (O2) to node {} (R2);
  \draw[->] (R2) to node {} (C2);
  \draw[->] (C2) to node {} (T2);
  \draw[->] (T2) to node {} (E2);
  \draw[->] (R1) to node [swap] {$|f_k|$} (R2);
  \draw[->] (C1) to node [swap] {$f_k$} (C2);
  \draw[->] (T1) to node [swap] {$\hat f_k$} (T2);
\end{tikzpicture}
\end{center}
Notice that $\T \simeq S^1 \subset \CC$. 
We denote by $\hat f(\theta)\subset (S^1)^A \subset \CC_*^A$ 
the vector with components $\hat f_k(\theta)$.
Assume that $f$ contains the constant monomial $1$, and consider the map
$\ord_B(f)\colon \RR^n \rightarrow \RR^m$ defined by
\begin{equation}
\label{eqn:explicitordermap}
 \ord_B(f)(\theta) = \Arg_\pi\big(\hat f(\theta)\big) \cdot B,
\end{equation}
where $\Arg_\pi$ denotes the componentwise principal argument map.
It was shown in \cite{FJ15} that the map $\ord_B(f)$ induces a map
\begin{equation} 
\label{eqn:ordermap}
 \ord_B(f)\colon \T^n\setminus\overline{\cL}_f \rightarrow \left\{\Arg_\pi(f)B + 2\pi \ZZ^m\right\}\cap \inter{\Z_B}.
\end{equation}
which in turn induces a bijection between the set
of connected components of the complement of $\overline{\cL}_f$
and the finite set in the right hand side of \eqref{eqn:ordermap}.
The map $\ord_B(f)$ is known as the \emph{order map} of the lopsided coamoeba.

\begin{remark}
The requirement that $f$ contains the monomial $1$ is related to the choice of branch cut of the
function $\Arg$; in order to obtain a well-defined map, we need the right hand side of 
\eqref{eqn:explicitordermap} to be discontinuous only for $\theta$ such that two components of
$\hat f(\theta)$ are antipodal, see \cite{FJ15}.
If $f$ does not contain the constant monomial $1$, then one should fix a point $\a_k\in A$
and multiply the vector $\hat f(\theta)$ by the scalar $\hat f_k(\theta)^{-1}$ before taking principal
arguments. It is shown in \cite[thm.\ 4.3]{FJ15} that the obtained map is independent of the choice of $\a_k$.
\end{remark}

If $\theta \in \T^n\setminus\overline{\cL}_f$, then we can choose $\varphi$ such that
\[
 \Re\left(e^{i\varphi} f_k(e^{i\theta})\right) > 0, \quad k=0,\dots, N-1.
\]
That is, the boundary of $\cL_f$ is contained in the hyperplane
arrangement consisting of all $\theta$ such that two
components
of $\hat f(\theta)$ are antipodal.

It has been conjectured that the number of connected components of 
the complement of $\overline{\cA}_f$ 
is at most $\Vol(A)$.\footnote{This 
conjecture has commonly been attributed to Mikael Passare,
however, it seems to originate from a talk given by Mounir Nisse
at Stockholm University in 2007.}
A proof in arbitrary dimension has been proposed by Nisse in \cite{Nis09},
and an independent proof in the case $n=2$ was
given in \cite{FJ14}. 
That the number of connected components
of the complement of $\overline{\cL}_f$
is at most $\Vol(A)$ follows from the theory of Mellin--Barnes integral
representations of $A$-hypergeometric functions,
see \cite{BFP11} and \cite{Beu11}.

A finite set $\I\subset \T^n$ which is in a bijective correspondence
with the set of connected components of the complement of
$\overline{\cA}_f$ by inclusion, will be said to be an \emph{index
set} of the coamoeba complement. This notation will be slightly abused;
a set $\I$ of cardinality $\Vol(A)$ will be said to be an index set
of the coamoeba if each connected component of its complement
contains exactly one element of $\I$. 

The term ``lopsided'' was first used by Purbhoo in \cite{P08},
denoting the corresponding condition to \eqref{eqn:colopsidedness} for amoebas:
the polynomial $f$ is said to be \emph{lopsided} at a 
point $x\in \RR^n$ if there is a $\a_k\in A$
such that the moduli $|f_k|(x)$ is greater than the sum of the remaining
modulis. As a comparison, note that the polynomial $f$ is 
colopsided at $\theta\in \T$ if and only if the greatest
intermediate angle of the components of $\hat f(\theta)$
is greater than the sum of the remaining intermediate angles.

\section{Real points and the coameoba of the $A$-discriminant}
\label{sec:RealPoints}
We will say that $f$ is \emph{real at $\theta$}, if there is a real subvector 
space $\ell\subset \CC$ such that $\hat f_k(\theta) \in \ell$ for all 
$k = 0,\dots, N-1$. If such a $\theta$ exist then $f$ is real,
that is, after a change of variables and multiplication with a Laurent monomial 
$f\in \RR_*^A$. In this section, we will study the function $\hat f$
from the viewpoint of real polynomials. 
Our main result is the following characterization
of the coamoeba of the $A$-discriminant of a circuit.

\begin{proposition}
\label{pro:discriminantmonomialarguments}
 Let $A$ be a nondegenerate circuit, and let $\delta_k$ be as in \eqref{eqn:Bcircuit}. Then, $\Arg(f)\in \cA_{\Delta}$ if and only if after possibly multiplying
 $f$ with a constant, there is a $\theta\in \RR^n$ such that $\hat f_k(\theta) = \delta_k$ for
 all $k$.
\end{proposition}

If $A$ is a circuit and $B$ is a Gale dual of $A$ then the Horn--Kapranov
parametrization of the reduced discriminant $\Delta_B$ can be lifted to 
a parametrization of the discriminant surface $\Delta$ as
\[
 z \mapsto \left(\b_0 z^{\a_0}, \dots, \b_{N-1}z^{\a_{N-1}}\right).
\]
Taking componentwise arguments, we obtain a simple proof
Proposition \ref{pro:discriminantmonomialarguments}.
In particular, the proposition can be interpreted as
a coamoeba version of the Horn--Kapranov parametrization valid
for circuits.
Our proof of Proposition \ref{pro:discriminantmonomialarguments}
will be more involved, however, 
for our purposes the lemmas contained in this section are of
equal importance.

\begin{lemma}
\label{lemma:realpointsduallattice}
Assume that the polynomial $f$ is real at $\theta_0\in \RR^n$. 
Then, $f$ is real at $\theta\in \RR^n$ 
if and only if $\theta \in \theta_0 + \pi L$, where $L$ is the dual lattice of $\ZZ A$.
\end{lemma}

\begin{proof}
 After translating $\theta$ and multiplying $f$ with a Laurent monomial, 
 we can assume that $\theta_0 = 0$,
 that $\ell_0 = \RR$, and that $f$ contains the monomial $1$. 
 That is, all coefficients of $f$ are real, in particular proving 
 \emph{if}-part of the statement.
 To show the \emph{only if}-part, notice first  that
 $\hat f (\theta) \subset \ell$ implies that $\ell$ contains 
 both the origin and $1$. That is, $\ell = \RR$.
 Furthermore, $\hat f (\theta)\subset \RR$ only if for each 
 $\a\in A$ there is a $k\in \ZZ$ such that $\<\a, \theta\> = \pi k$,
 which concludes the proof.
\end{proof}

The $A$-discriminant $\Delta$ related to a circuit
has been described in 
\cite[chp.\ 9, pro.\ 1.8]{GKZ94} where the formula
\begin{equation}
\label{eqn:CircuitDiscriminant}
 \Delta(f) = \prod_{\delta_k = 1} \b_k^{\b_k}\,\prod_{\delta_k = -1} f_k^{-\b_k} - \prod_{\delta_k = -1}\b_k^{-\b_k}\,\prod_{\delta_k = 1}f_k^{\b_k}
\end{equation}
was obtained. In particular, $\Delta$ is a binomial.
As the zonotope $\Z_B$ is a symmetric interval of length $2\pi\Vol(A)$,
the image of the map $\ord_B(f)$ is of cardinality $\Vol(A)$ unless
\begin{equation} 
\label{eqn:translationinzonotope}
 \Arg_\pi(f)B \equiv 2\pi\Vol(A) \mod 2\pi.
\end{equation}
In particular, the complement of $\overline{\cA}_f$ has the maximal 
number of connected
components (i.e., $\Vol(A)$-many) unless the equivalence 
\eqref{eqn:translationinzonotope} holds.

\begin{lemma}
\label{lem:simplexequalarguments}
For each $\kappa = 0,1, \dots, n+1$,
there are exactly $\Vol(A_\kappa)$-many points $\theta \in \T$ such that
\begin{equation}
\label{eqn:simplexequalarguments}
\hat f_k(\theta) = \delta_k, \quad \forall\,k\neq \kappa.
\end{equation}
\end{lemma}

\begin{proof}
By applying an integer affine transformation, the statement follows from the case
when $A_\kappa$ consist of the vertices of the standard simplex.
\end{proof}

\begin{lemma}
\label{lem:uniquelastarguments}
Fix $\kappa \in \{0,\dots, n+1\}$. 
For each $\theta$ fulfilling \eqref{eqn:simplexequalarguments}, 
let $\varphi_\theta \in \T$ be defined by the condition that 
if $\arg_\pi(f_{\kappa})= \varphi_\theta$ then
\begin{equation}
\label{eqn:lastargumentdefinedbytheta}
 \hat f_{\kappa}(\theta) = \delta_\kappa.
\end{equation}
Assume that $\ZZ A = \ZZ^n$. Then, the numbers $\varphi_\theta$ are distinct.
\end{lemma}

\begin{proof}
We can assume that $\a_0 = \0$ and that $f_0 = 1$.
Assume that $\varphi_{\theta_1} = \varphi_{\theta_2}$.
Then,
\[
 \<\a, \theta_2\> = \<\a, \theta_1\> + 2\pi r, \quad \forall \a\in A.
\]
By translating, we can assume that $\theta_1 = 0$, and hence, since $1$ 
is a monomial of $f$,
that all coefficients are real. 
Consider the lattice $L$ consisting of all points 
$\theta \in \RR^n$ such that $f$ is real at $\theta$.
Since $\ZZ A = \ZZ^n$, Lemma~\ref{lemma:realpointsduallattice} shows
that $L =\pi\ZZ^n$. However, we find that
\[
 \Big\<\a, \frac{\theta_2}2\Big\>  = \pi r,
\]
and hence $\frac{\theta_2}2 \in L$. This implies that $\theta_2 \in 2\pi\ZZ^n$,
and hence $\theta_2 = 0$ in $\T^n$.
\end{proof}

\begin{proof}[Proof of Proposition~\ref{pro:discriminantmonomialarguments}]
 Assume first that there is a $\theta$ as in the statement of the proposition,
 where we can assume that $\theta = 0$.
 Then, $\arg(f_k) = \arg(\delta_k)$.
 It follows that the monomials
 \[
  \prod_{\delta_k = 1} \b_k^{\b_k}\,\prod_{\delta_k = -1} f_k^{-\b_k}
  \quad\text{and}\quad
  \prod_{\delta_k = -1} \b_k^{\b_k}\,\prod_{\delta_k = 1} f_k^{-\b_k}
 \]
 have equal signs. Therefor, $\Delta$ vanishes 
 for $f_k = \delta_k|\b_k|$, implying that $\Arg(f)\in \cA_{\Delta}$.
 
 For the converse, fix $\kappa$, and reduce $f$
 by requiring that $f_k = \delta_k |\b_k| = \b_k$ for $k\neq \kappa$.
 Let $\I$ denote the set of points $\theta \in \T^n$ such that 
 $\hat f_k(\theta) = \delta_k$ for $k\neq \kappa$,
 which by Lemma~\ref{lem:simplexequalarguments} has cardinality $V_\kappa$. 
 By Lemma~\ref{lem:uniquelastarguments}, the set $\I$
 is in a bijective correspondence with values of 
 $\arg(f_\kappa)$ such that 
 $\hat f_\kappa (\theta) = \delta_\kappa$. 
 Therefor, we find that $\Delta$ vanishes at $f_\kappa = V_\kappa e^{i\varphi}$ for
 each $\varphi\in \I$.
 However, the discriminant $\Delta$ specializes, up to a constant, to the binomial
 \[
  \Delta_\kappa(f_\kappa) = f_\kappa^{|\b_\kappa|} - \b_{\kappa}^{|\b_\kappa|} = f_\kappa^{V_\kappa} - \b_{\kappa}^{V_\kappa},
 \]
 which has exactly $V_\kappa$-many solutions in $\CC_*$ of distinct arguments. 
 Hence, since $\Delta(f) = 0$ by assumption,
 and comparing the number of solutions, it holds that $\hat f_\kappa (\theta) = \delta_\kappa$ for one of the points
 $\theta\in \I$.
\end{proof}

\section{The space of coamoebas}
\label{sec:configurationspace}

Let $U_k\subset\CC_*^A$ denote the set of all $f$ such that 
the number of connected components 
of the complement of $\overline{\cA}_f$ is $\Vol(A)-k$. 
Describing the sets $U_k$ is known as the problem of describing the \emph{space of coamoebas}
of $\CC_*^A$.
In this section, we will give explicit descriptions of the sets $U_k$ in the case when $A$ is a circuit.
As a first observation we note that the image of the map $\ord_B(f)$ is
at least of cardinality $\Vol(A)-1$, implying that
\[
 \CC_*^A = U_0 \cup U_1,
\]
and in particular $U_k = \emptyset$ for $k \geq 2$. Hence, it suffices for us
to give an explicit description of the set $U_1$.
Our main result is the following two theorems, highlighting also the
difference between vertex circuits and simplex circuits.
Note that $\Delta$ is a real polynomial \cite{GKZ94}.

\begin{theorem}
\label{thm:configurationspaceinteriorpoint}
Assume that $A$ is a nondegenerate simplex circuit, with $\a_{n+1}$ as an interior point. 
Choose $B$ such that $\delta_{n+1} =-1$,
and let $\Delta$ be as in 
\eqref{eqn:CircuitDiscriminant}. Then, $f\in U_1$ if and only if
$\Arg(f) \in \cA_{\Delta}$ and
\begin{equation}
\label{eqn:SpaceInequality}
(-1)^{\Vol(A)}\,\Delta\big(\delta_0|f_0|, \dots, \delta_{n+1}|f_{n+1}|\big) \leq 0.
\end{equation}
\end{theorem}

\begin{theorem}
\label{thm:configurationspacevertex}
Assume that $A$ is a vertex circuit.
Then, $f\in U_1$ if and only if $\Arg(f) \in \cA_{\Delta}$.
\end{theorem}

The article \cite{TdW13} describes the 
\emph{space of amoebas} in the case when $A$ is a simplex circuit in
dimension at least two.
In this case, the number of connected components of the amoeba complement
is either equal to the number of vertices of $\N_A$ or
one greater. One implication of \cite[thm.\ 4.4 and thm.\ 5.4]{TdW13}
is that, if the amoeba complement has the minimal number
of connected components, then 
\[
(-1)^{\Vol(A)}\,\Delta\big(\delta_0|f_0|, \dots, \delta_{n+1}|f_{n+1}|\big) \geq 0.
\]
Furthermore, this set intersect $U_1$ only in the discriminant
locus $\Delta(f) = 0$.
The space of amoebas in the case when $A$ is a simplex circuit in
dimension $n=1$ has been studied in \cite{TdW14}, and is a more
delicate problem. On the other hand, if $A$ is a vertex circuit, 
then each $f\in \CC_*^A$ is maximally sparse and hence has
a solid amoeba. That is, the components of the complement of the amoeba
is in a bijective correspondance with the vertices of the Newton
polytope $\N_A$. In particular, the number of connected components 
of the amoeba complement does not depend on $f$. From Theorems 
\ref{thm:configurationspaceinteriorpoint} and \ref{thm:configurationspacevertex}
we see that a similar discrepancy between simplex circuits and vertex circuits
occur for coamoebas.

\begin{example}
\label{ex:configurationspaces}
The reduced family
\[
 f(z) = 1 + z_1^3 + z_2^3 + \xi\,z_1 z_2.
\]
was considered in \cite[ex.\ 6, p.\ 59]{R03}, 
where the study of the space of amoebas was initiated. 
We have drawn the space of amoebas and coamoebas jointly in 
the left picture in Figure~\ref{fig:configurationspace}.
The blue region, whose boundary is a hypocycloid, 
marks values of $\xi$ for which the amoeba complement
has no bounded component. The set $U_1$ is seen in orange. The red dots
is the discriminant locus $\Delta(\xi) = 0$, which is contained in the circle $|\xi| = 3$
corresponding an equality in \eqref{eqn:SpaceInequality}.
\end{example}

\begin{example}
\label{ex:configurationspaces2}
The reduced family
\[
f(z) = 1 + z_1 + z_2^3 + \xi\,z_1^3 z_2
\]
is a vertex circuit. In this case, the topology of the amoeba does not depend on the coefficient $\xi$. 
The space of coamoebas is drawn in the right picture in Figure~\ref{fig:configurationspace}.
The set $U_1$ comprises the three orange lines emerging from the origin. 
The red dots is the discriminant locus $\Delta(\xi) = 0$.
It might seem like the set $U_0$ is disconnected, however this a consequence
of that we consider $f$ in reduced form. In $\CC_*^A$ the set $U_0$
is connected, though not simply connected.
\end{example}

\begin{figure}[h]
\centering
\includegraphics[width=30mm]{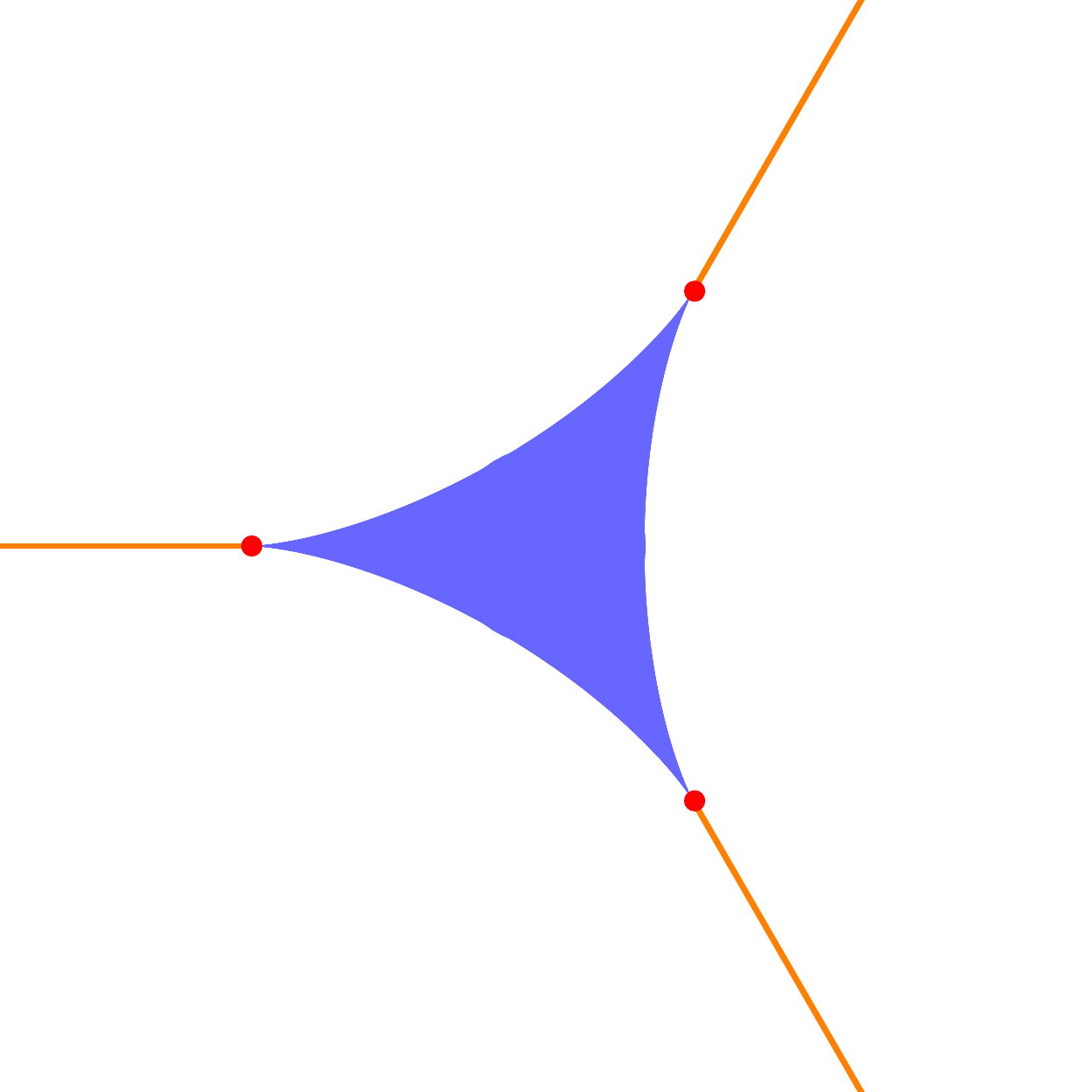}
$\quad\quad$
\includegraphics[width=30mm]{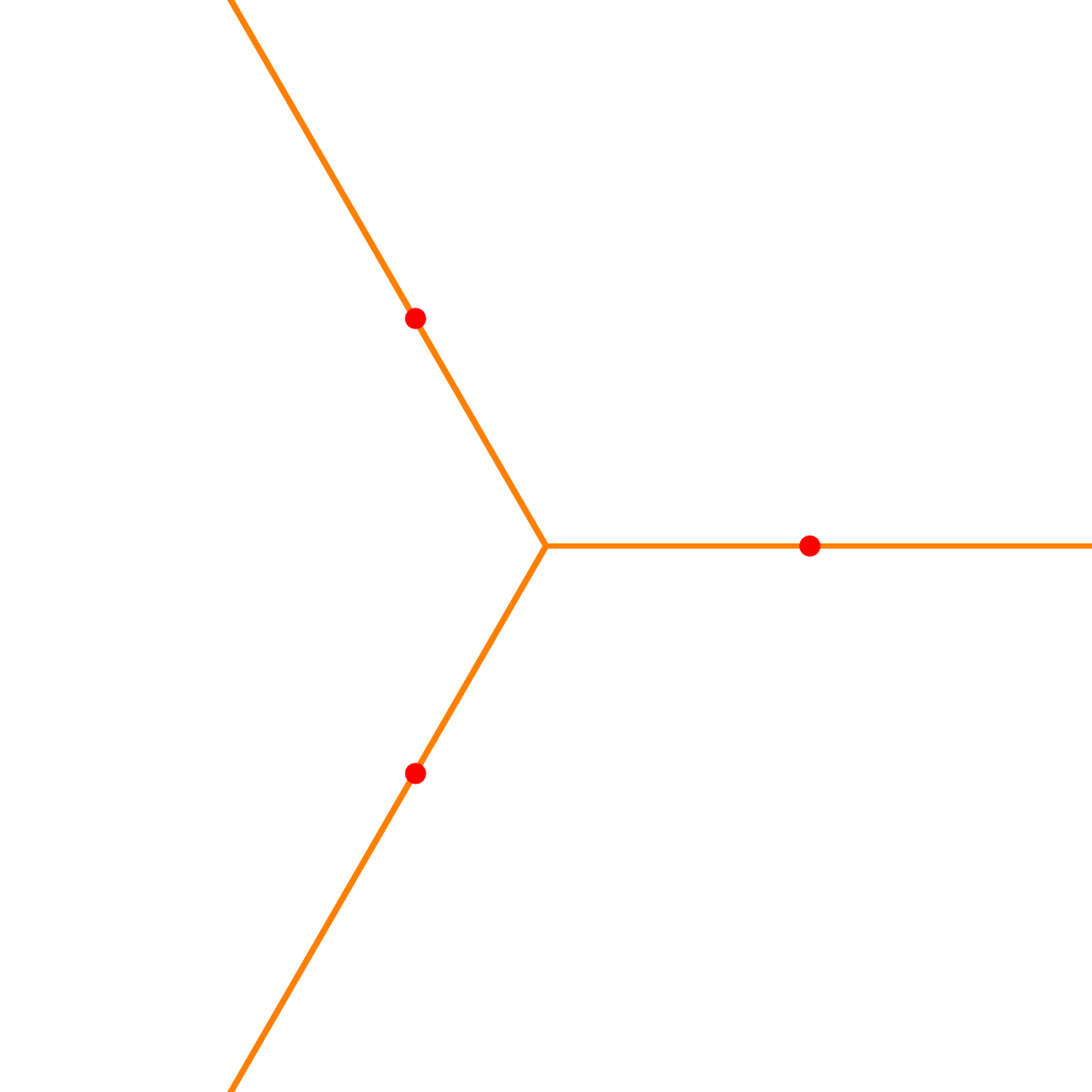}
\caption{The amoeba and coamoeba spaces of Examples~\ref{ex:configurationspaces}
and~\ref{ex:configurationspaces2}.}
\label{fig:configurationspace}
\end{figure}

\subsection{Proof of Theorem~\ref{thm:configurationspaceinteriorpoint}}
Impose the assumptions of Theorem \ref{thm:configurationspaceinteriorpoint}.
Then,
\[
 \b_0 +\dots + \b_n = -\b_{n+1} = \Vol(A),
\]
and in particular $V_{n+1} = \Vol(A)$.
By Lemma~\ref{lem:simplexequalarguments} there is a set $\I$
of cardinality $\Vol(A)$ consisting of all points $\theta$
such that $\hat f_{0}(\theta) = \dots = \hat f_{n}(\theta) = \delta_k = 1$
In particular, $f$ is colopsided at $\theta\in \I$
unless $\hat f_{n+1}(\theta) = -1$. 
It was shown in \cite[sec.\ 5]{FJ15} that, if $f\not\in\cA_{\Delta}$, 
then $\I$ is an index set for the complement of $\overline{\cA}_f$. 
In fact, $\I$ is an index set of the complement of $\overline{\cA}_f$
for arbitrary $f$.

\begin{proposition}
\label{pro:testingtheta}
Let $A$ be a simplex circuit. Assume that 
$\Arg(f) \in \cA_{\Delta}$, i.e., that there exists a $\theta\in \I$
with $\hat f_{{n+1}}(\theta) = \delta_{n+1}$.
Then, the complement of $\overline{\cA}_f$ has 
$\Vol(A)$-many connected components 
if and only if it contains $\theta$.
\end{proposition}

\begin{proof}
We can assume that $\theta = 0$.
To prove the \emph{if}-part, assume that $0 \in \Theta$ for some connected component $\Theta$
of the complement of $\overline{\cA}_f$. We wish to show that $f$ is never colopsided
in $\Theta$, for this implies that the complement of $\overline{\cA}_f$ 
has $\Vol(A)$-many connected components.
Assume to the contrary that there exist a 
point $\hat \theta\in \Theta$ such that $f$ is colopsided
at $\hat\theta$. Then, $\ord_B(f)(\hat\theta) = m\pi$ for some
integer $m$, with $|m|<\Vol(A)$, see \eqref{eqn:ordermap}. 
Let $f^\varepsilon = (f_0, \dots, f_n, f_{n+1}e^{i\varepsilon})$.
Then $f^\varepsilon$ is colopsided at $0$ for $\varepsilon \notin 2\pi\ZZ$.
By continuity of roots, for $\varepsilon > 0$ sufficiently small, 
the points $0$ and $\hat\theta$ are contained in the same connected
component of the complement of $\overline{\cA}(f^\varepsilon)$.
Hence, by \cite[pro.\ 3.9]{FJ15}, they are 
contained in the same connected component of 
the complement of $\overline{\cL}(f^\varepsilon)$. However,
\[
\ord_B(f^\varepsilon)(0) = \Vol(A)(\pi - \varepsilon) \neq m(\pi - \varepsilon)= \ord_B(f^\varepsilon)(\hat \theta),
\]
contradicting that $\ord_B(f^\varepsilon)$ is constant on connected
components of the complement of $\cL(f^\varepsilon)$.

To prove the \emph{only if}-part, assume that there exists a 
connected component $\Theta$ of
the complement of $\overline{\cA}_f$ in which $f$ is never colopsided.
We wish to prove that $0\in \Theta$.
As $f^\varepsilon$ is colopsided at $0$ for $\varepsilon > 0$ sufficiently
small, we find that $0 \in \overline\Theta$. Indeed, if this was
not the case, then the complement of $\overline{\cA}(f^\varepsilon)$ 
has $(\Vol(A)+1)$-many connected components,
a contradiction.
As $0 \not\in \H_f$, and by \cite[lem.\ 2.3]{FJ15},
there exists a disc $D_0$ around $0$ such that 
\[
 D_0\cap(\T^n\setminus\overline{\cA}_f) = D_0\cap\Theta.
\]
Furthermore, $D_0\cap \Theta \neq \emptyset$, since $0 \in \overline{\Theta}$.
Let $\theta\in D_0\cap \Theta$. 
As $f$ is a real polynomial, conjugation yields that 
$-\theta \in D_0\cap\Theta$.
However, $\Theta \subset \RR^n$ is convex, implying that
$0 \in \Theta$.
\end{proof}

\begin{proof}[Proof of Theorem~\ref{thm:configurationspaceinteriorpoint}]
If $\Arg(f)\not\in\cA_{\Delta}$ then the image of $\ord_B(f)$ is of cardinality
$\Vol(A)$, and hence $f\in U_0$.
Thus, we only need to consider $\Arg(f)\in\cA_{\Delta}$, 
where we can assume that $\hat f(0) = \delta_k$ for all $k$.
In particular, $f$ is a real polynomial.
By Proposition~\ref{pro:testingtheta}, it holds that the complement of
$\overline{\cA}_f$ has $\Vol(A)$-many connected
components if and only if it contains $0$.
Keeping $f_0, \dots, f_n$ and $\arg(f_{n+1})$ fixed, let us consider
$f$ as a function of $|f_{n+1}|$.
As $f$ is a real polynomial, 
it restricts to a map $f\colon \RR_{\geq 0}^n \rightarrow \RR$,
whose image depends nontrivially on $|f_{n+1}|$.
Notice that $0\in \overline{\cA}_f$ if and only if
$f(\RR_{\geq0}^n)$ contains the origin.
Since $\hat f_k(0) = \delta_k = 1$ for $k\neq n+1$, and
since $\a_{n+1}$ is an interior point of $A$,
the map $f$ takes the boundary of $\RR_{\geq 0}^n$ to $[1, \infty)$.
In particular, if $0\in f(\RR_{\geq0}^n)$, 
then $0\in f(\RR_{+}^n)$.
The boundary of the set of all $|f_{n+1}|$ 
for which $0\in f(\RR_{\geq0}^n)$ is the set of all values of $|f_{n+1}|$ for 
which $f(\RR_{+}^n) = [0,\infty)$. 
Furthermore, $f(\RR_+^n) = [0,\infty)$ holds
if and only if there 
exists an $r\in \RR_+^n$ such that $f(r) = 0$, while $f(r) \geq 0$ in a 
neighborhood of $r$, implying that $r$ is a critical point of $f$.
That is, 
\[
\Delta\big(\delta_1|f_1|, \dots, \delta_{n+1}|f_{n+1}|\big) = 0.
\]
Since $\Delta$ is a binomial, there is exactly one such value of
$|f_{n+1}|$.
Finally, we note that $0\in\overline{\cA}_f$ if $|f_{n+1}|\rightarrow \infty$, 
which finishes the proof.
\end{proof}

\subsection{Proof of Theorem~\ref{thm:configurationspacevertex}}
If $\Arg(f)\not\in \cA_{\Delta}$, then the image of $\ord_B(f)$ is of 
cardinality $\Vol(A)$
and hence $f\in U_0$.
Assume that $\Arg(f)\in \cA_{\Delta}$, and that 
$\hat f_k(0) =\delta_k$ for all $k$.
It holds that $0\in \H_f$ 
since there exists two adjacent vertices $\a_0$ and $\a_1$ of $A$ such that 
$\delta_{0} = 1$ and $\delta_{1} = -1$. Let,
\begin{align*}
H & = \{\theta\,|\, \<\a_{0}-\a_{1},\theta\> =  0\}\\
\intertext{be the hyperplane of $\H_f$ containing $0$.
Assume that exists connected component $\Theta$ of the complement of
$\overline{\cA}_f$ in which $f$ is nowhere colopsided.
As in the proof of Proposition~\ref{pro:testingtheta}, we conclude that
$0\in \overline{\Theta}$, for otherwise we could construct a coamoeba
with $(\Vol(A)+1)$-many connected components of its complement.
As $H\subset \overline{\cA}_f$, we find that $\Theta$
is contained in one of the half-spaces}
H_\pm & = \{\theta\,|\, \pm\<\a_0-\a_1,\theta\> > 0\},\\
\intertext{say that $\Theta \subset H_+$.
Let $f^\varepsilon = (f_0e^{i\varepsilon}, f_1, \dots, f_{n+1})$,
and let $H^\varepsilon$ denote the corresponding hyperplane}
H^\varepsilon &= \{\theta\,|\, \<\a_0-\a_{1},\theta\> =  -\varepsilon\}.
\end{align*}
For $|\varepsilon|$ sufficiently small, continuity of roots 
implies that there is a connected component
$\Theta^\varepsilon\subset H_+^\varepsilon$ in which $f^\varepsilon$
is never colopsided.
However,
by choosing the sign of $\varepsilon$, we can force $0\in H_-^\varepsilon$.
This implies that the coamoeba $\overline{\cA}_{f^\varepsilon}$ has
$(\Vol(A)+1)$-many connected components of its complement, a contradiction.

\section{The maximal area of planar circuit coamoebas}
\label{sec:maximalarea}
In this section, we will prove the following bound.

\begin{theorem}
\label{thm:areabound}
 Let $A$ be a planar circuit, and let $f\in \CC_*^A$. Then $\Area(\cA_f) \leq 2\pi^2$.
\end{theorem}

Furthermore, we provide the following classification of
for which circuits the bound of Theorem~\ref{thm:areabound}
is sharp.

\begin{theorem}
\label{thm:unimodularmaximalthm}
Let $A$ be a planar circuit. Then there exist
a polynomial $f\in \CC_*^A$ such that $\Area(\cA_f) = 2\pi^2$ 
if and only if $A$ admits an equimodular triangulation.
\end{theorem}

\begin{example}
\label{ex:maximalareasquare}
Let $f(z) = 1 + z_1 + z_2 - rz_1 z_2$ for $r\in \RR_+$. Notice that $A$ admits a
unimodular triangulation. The shell $\H_f$ consist of 
the families $\theta_1 = k_1 \pi$ and $\theta_2 = k_2\pi$ for $k_1,k_2\in \ZZ$. 
Hence, the shell $\H_f$ divides $\T^2$ into four regions of equal area.
Exactly two of these regions are contained in the 
coamoeba, which implies that $\Area(\cA_f) = 2 \pi^2$. See the left picture
of Figure~\ref{fig:maximalareaexamples}.
\end{example}

\begin{example}
\label{ex:maximalareainterior}
Let $f(z) = 1+z w^2 + z^2 w - r z w$ for $r\in \RR_+$. Also in this case $A$ admits
a unimodular triangulation. Notice that $\Arg(f)\in \cA_{\Delta}$. 
The coamoeba of the trinomial $g(z) = 1+z w^2 + z^2 w$ has three
components of its complement, of which $f$ is colopsided in two. 
We have that $\H_f= \H_g$. Thus, if the 
complement of $\overline{\cA}_f$ has two connected components, 
i.e.\ if $r \geq 3$,
then one of the three components of the  
complement of $\overline{\cA}_g$
is contained in $\overline{\cA}_f$,
which in turn implies that $\Area(\cA_f) = 2\pi^2$.
See the right picture of Figure~\ref{fig:maximalareaexamples}.
\end{example}

\begin{figure}[h]
\centering
\includegraphics[height=30mm]{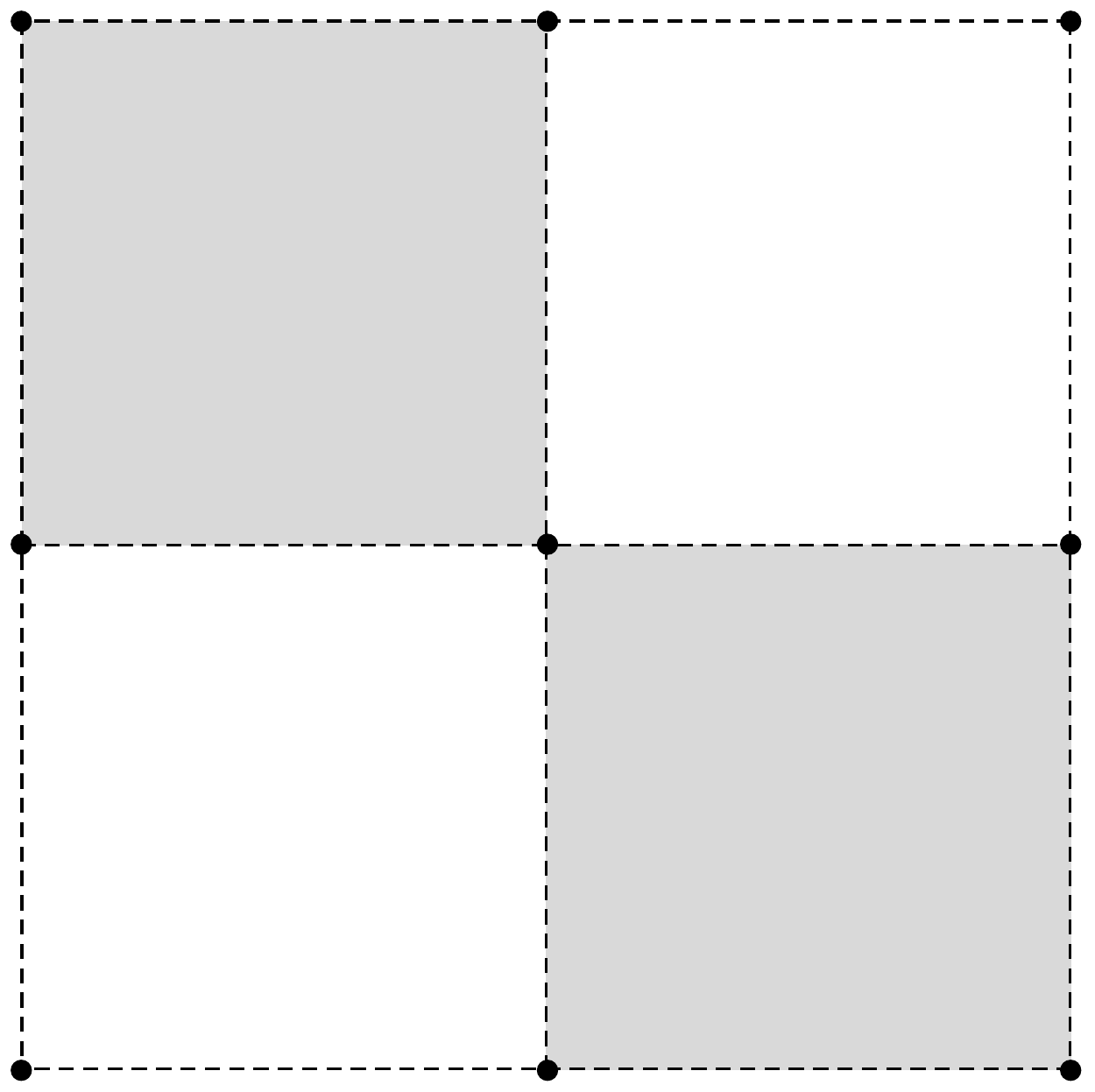}
\hspace{16mm}
\includegraphics[height=30mm]{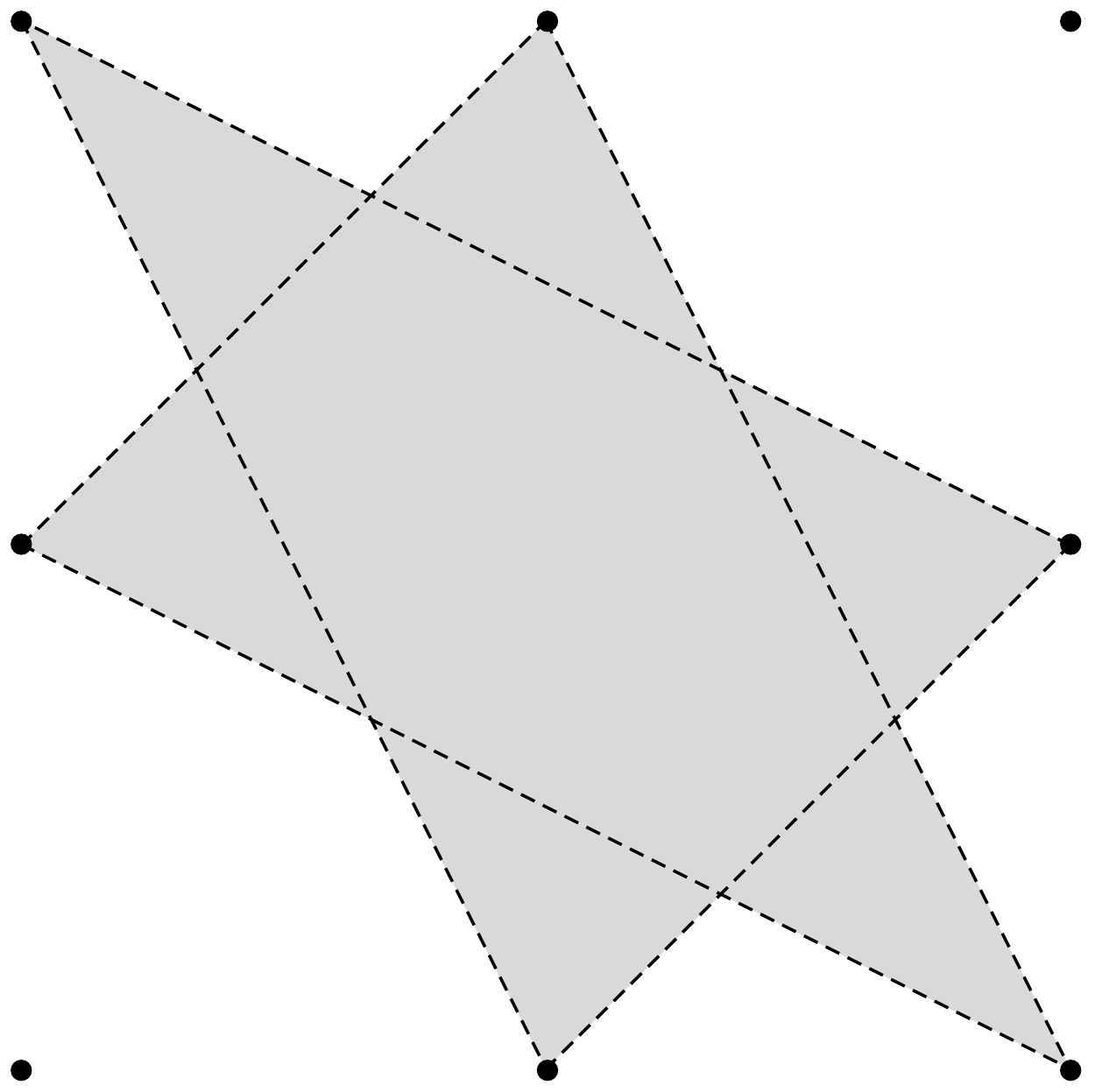}
\caption{The coamoebas of Examples \ref{ex:maximalareasquare} and \ref{ex:maximalareainterior}.}
\label{fig:maximalareaexamples}
\end{figure}

Let us compare our results to the corresponding results of planar circuit amoebas. 
It was shown in \cite[thm.\ 12, p.\ 30]{R03} that the sharp upper bound on the number
of connected components of the amoeba complement is $\# A$. 
In \cite{MR01}, a bound on the area of the amoeba was
given as $\pi^2 \Vol(A)$, and it was shown that maximal area was obtained
for Harnack curves. 
For coamoebas, to roles of the integers $\Vol(A)$ and $\# A$ are reversed.
The upper bound on the number of 
connected components of the coamoeba complement is given by $\Vol(A)$.
While, at least for codimension $m\leq 1$,
the maximal area of the coameoba is $\pi^2(m+1) = \pi^2(\# A - n)$. 
Note also that the coamoebas of Examples~\ref{ex:maximalareasquare} and \ref{ex:maximalareainterior} 
are both coamoebas of Harnack curves.

Consider a bivariate trinomial $f$, with one marked monomial. 
Let $\Sigma = \Sigma(f)$ denote the quadruple of polynomials obtained
by flipping signs of the unmarked monomials. Furthermore, let 
\[
 \H_{\Sigma} = \bigcup_{g\in \Sigma} \H_g,
\]
which is a hyperplane arrangement in $\RR^2$ (or $\T^2$).
Let $\P_{\Sigma}$ denote the set of all intersection points of 
distinct hyperplanes in $\H_{\Sigma}$.
\begin{proposition}
\label{pro:unioncovers}
 Let $f(z)$ be a bivariate trinomial. 
 Then, the union
 \[
  \overline{\cA}_{\Sigma} = \bigcup_{g\in\Sigma} \overline{\cA}_g,
 \]
 covers $\RR^2$. To be specific, $\P_{\Sigma}$ is covered thrice, 
 $\H_{\Sigma}\setminus \P_{\Sigma}$ is covered twice, and
 $\RR^2\setminus \H_{\Sigma}$ is covered once.
\end{proposition}

\begin{proof}
 After applying an integer affine transformations, 
 we reduce to the case when $A$ consist of the vertices of the standard 
 simplex. This case that follows from the
 description in \cite{FJ15} and \cite{NS13}, see also Figure~\ref{fig:simplexcoamoeba}.
\end{proof}

\begin{corollary}
If $f(z)$ is a bivariate trinomial, then $\Area(\overline{\cA}_f) = \pi^2$.
\end{corollary}

\begin{proof}
 The coamoebas appearing in the union $\overline{\cA}_{\Sigma}$, when considered in $\RR^2$,
 are merely translations of eachother. Hence, they have equal area. As
 they cover the torus once a.e., 
 and $\Area(\T^2) = 4 \pi^2$, the result follows. 
\end{proof}

Notice that a bivariate trinomial is not supported on a circuit,
but on the vertex set of a simplex.
Let $f_{\hat k}$ denote the image of $f$ under the projection
$\pr_k\colon \CC_*^A \rightarrow \CC_*^{A_k}$.
As shown in \cite{FJ15} the family of trinomials $f_{\hat k}$, $k=1, \dots, 4$,
contains all necessary information about the lopsided coamoeba $\cL_f$.

\begin{lemma}
\label{lem:twocolopsidedtrinomial}
 Let $A$ be a planar circuit, and let $f\in \CC_*^A$.
 Assume that $\theta\in\T$ is generic in the sense that no two 
 components of $\hat f(\theta)$ are antipodal,
 and assume furthermore that $f$ is \emph{not} colopsided at $\theta$.
 Then, exactly two of the trinomials $f_{\hat 1}, \dots, f_{\hat{4}}$ are colopsided at $\theta$.
\end{lemma}

\begin{proof}
 Fix an arbitrary point $\a_1\in A$, and let $\ell\subset \CC$ denote the real 
 subvector space containing $\hat f_{1}(\theta)$. 
 As $f$ is not colopsided at $\theta$,
 both half spaces relative $\ell$ contains at least one component of $\hat f(\theta)$.
 There is no restriction to assume that the upper half space contains the two components
 $\hat f_{2}(\theta)$ and $\hat f_{3}(\theta)$,
 and that the latter is of greatest angular distance from $\hat f_{1}(\theta)$.
 Then, $f_{\hat 4}$ is colopsided at $\theta$. Furthermore,
 we find that $f_{\hat 2}$ is not colopsided at $\theta$, for if it where then 
 so would $f$.
 As $\a_1\in A_4$ and $\a_1\in A_2$, there is at least one trinomial 
 obtained from
 $f$ containing $\a_1$ which is not colopsided at $\theta$, 
 and at least one which is
 colopsided at $\theta$. As $\a_1$ was arbitrary, it follows that exactly two of the trinomials
$f_{\hat 1}, \dots, f_{\hat4}$ are colopsided at $\theta$, and exactly two are not.
\end{proof}

\begin{proof}[Proof of Theorem \ref{thm:areabound}]
 By containment, it holds that $\Area(\cA_f)\leq \Area(\cL_f)$, and
 thus it suffices to calculate the area of $\cL_f$.
 By \cite[pro. 3.4]{FJ15}, we have that
 \begin{equation}
 \label{eqn:unionoftrinomials}
  \cL_f = \bigcup_{k=1}^4 \cA_{f_{\hat k}}.
 \end{equation}
 For a generic point $\theta\in \cL_f$, Lemma~\ref{lem:twocolopsidedtrinomial} 
 gives
 that $\theta$ (and, in fact, a small neighborhood of $\theta$) is contained in the 
 interior of exactly two out of the four coamoebas in the
 right hand side of \eqref{eqn:unionoftrinomials}. Hence,
 \[
  \Area(\cL_f) = \frac12 \big(\Area(\cA_{f_{\hat 1}})+ \dots +\Area(\cA_{f_{\hat 4}})\big) = 2\pi^2.\qedhere
 \]
\end{proof}

\begin{proof}[Proof of Theorem \ref{thm:unimodularmaximalthm}]
 To prove the \emph{if} part, we will prove that $A$ admits an equimodular 
 triangulation only if, after applying an integer affine transformation,
 it is equal to the point configuration of either 
 Example~\ref{ex:maximalareasquare} or Example~\ref{ex:maximalareainterior}.
 Assume that $\a_1, \a_2$, and $\a_3$ are vertices of $\N_A$. 
 After applying an integer affine transformation, 
 we can assume that 
 $\a_1 = k_1 \e_1$, that $\a_2 = k_2\e_2$ with $k_1 \geq k_2$, 
 and that $\a_3 = \0$. Notice that such
 a transformation rescales $A$, 
 though it does not affect the area of the 
 coamoeba $\cA_f$ \cite{FJ15}.
 Let $\a_4 = m_1 \e_1 + m_2 \e_2$.
 
 If $A$ is a vertex circuit, then each triangulation of $A$ consist of
 two simplices, which are of equal area by assumption. Comparing the areas of the 
 subsimplices of $A$, we obtain the relations 
 \[
  |k_1 k_2 - k_1m_2 - k_2 m_1| = k_1 k_2 \quad \text{and} \quad k_1 m_2 = k_2 m_1.
 \]
 In $m$, this system has $(k_1, k_2)$ as the only nontrivial solution,
 and we conclude that $A$ is the unit square, up to integer affine 
 transformations.
 
 If $A$ is a simplex circuit, then $A$ has one triangulation with
 three simplices of equal area. Comparing areas, we obtain the relations
 \[
  3 k_1 m_2 = 3 k_2 m_1 = k_1 k_2.
 \]
 Thus, $3m_1 = k_1$ and $3 m_2 = k_2$, and we conclude that $A$ is the 
 simplex from Example~\ref{ex:maximalareainterior}, 
 up to integer affine transformations.

To prove the \emph{only if}-statement,
consider $f\in \CC_*^A$.
Let $S = \{\a_1, \a_2\}\subset A$ be such that the line segment 
$[\a_1, \a_2]$ is interior to $\N_A$.
Applying an integer affine transformation, we can assume that $[\a_1, \a_2]\subset\RR\e_1$,
and that $\a_3$ and $\a_4$ lies in the upper and lower half space respectively. 
Then, the hyperplane arrangement $\cA_{f_S}\subset \T$ consist of 
$\Len[\a_1, \a_2]$-many lines, each parallel to the $\theta_2$-axis.
If $\a_3 = m_{31}\e_1 + m_{32}\e_2$ and $\a_4 = m_{41} \e_1 + m_{42}\e_2$,
then $\hat f_{3}(\theta)$ and $\hat f_{4}(\theta)$ takes
$m_{32}$ respectively $m_{42}$ turns around the origin when $\theta$
traverses once a line of $\cA_{f_S}$.
Notice that $\cA_{f_S}\subset \overline{\cL}_f$,
as $\hat f_{1}(\theta)$ and $\hat f_{2}(\theta)$ are antipodal
for $\theta \in \cA_{f_S}$. That is, for such $\theta$,
$\hat f_S(\theta)$ is contained
in a real subvector space $\ell_\theta\subset \CC$.

Assume that $f$ is colopsided for some $\theta\in\cA_{f_S}$,
so that in particular $\theta \not\in \cA_f$.
If $\theta \in \H_f$, then at exactly one of the points 
$\hat f_{3}(\theta)$ and $\hat f_{4}(\theta)$ is contained in $\ell_\theta$,
for otherwise $f$ would not be colopsided at $\theta$.
By wiggling $\theta$ in $\cA_{f_S}$ we
can assume that $\theta \not \in \H_f$.
Under this assumption, we find that $\theta\notin \overline{\cA}(f)$. 
Thus, there is a neighborhood $N_\theta$ which is separated 
from $\overline{\cA}_f$.  
As $\theta\in \overline{\cL}_f$, the intersection 
$N_\theta \cap \overline{\cL}_f$
has positive area, implying that $\Area(\overline{\cA}_f) < \Area(\overline{\cL}_f)$.

Thus, if $f$ is such that $\Area(\overline{\cA}_f) = 2\pi^2$,
then $f$ can never be colopsided in $\cA_{f_S}$. In particular,
for $\theta\in\cA_{f_S}$ such that $\hat f_{3}(\theta)\in \ell$, 
it must hold $\hat f_{4}(\theta)\in \ell$, and vice versa.
As there are $2m_{32}$ points of the first kind, and $2m_{42}$ points of
the second kind, it holds that $m_{32}=m_{42}$. Hence, the simplices
with vertices $\{\a_1, \a_2, \a_3\}$ and $\{\a_1, \a_2, \a_4\}$ have equal area. 

If $A$ is a vertex cicuit, this suffices in order to conclude that
$A$ admits an equimodular triangulation. 
If $A$ is a simplex circuit, then we can assume that  
$\a_1$ is an interior point of $\N_A$. Repeating the
argument for either $S=\{\a_1,\a_3\}$ or 
$S = \{\a_1, \a_4\}$ yields that $A$ has a triangulation with
three triangles of equal area. That is, it admits an equimodular triangulation.
\end{proof}

\section{Critical points}
\label{sec:criticalpoints}
Let $C(f)$ denote the critical points of $f$, that is, the variety defined by \eqref{eqn:criticalpoint}.
Let $\I = \Arg(C(f))$ denote the coamoeba of $C(f)$. We will say that $\I$ is the set of \emph{critical arguments}
of $f$.
In this section we will prove that, under certain assumptions on $A$,
the set $\I$ is an index set of the coamoeba complement.
That it is necessary to impose assumptions on $A$ 
is related to the fact that an integer affine transformation
acts nontrivially on the set of critical points $C(f)$.

Let $A$ be a circuit, with the elements $\a\in A$ ordered so that
it has a Gale dual  $B = (B_1, B_2)^t$ such that $B_1\in \RR_+^{m_1+1}$ 
and that $B_2 \in \RR_-^{m_2+1}$.
That is, $B_1$ has only positive entries, while $B_2$ has only negative entries.
We have that $m_1 + m_2  = n$. 
Let $A = (A_1, A_2)$ denote the corresponding decomposition
of the matrix $A$.
We will say that $A$ is in \emph{orthogonal form} if
\begin{equation}
\label{eqn:Aunionofsimplicesform}
   A = \left(\begin{array}{cc} 1 & 1 \\ \tilde A_1 & 0 \\ 0 & \tilde A_2 \end{array}\right),
\end{equation}
where $\tilde A_1$ is an $m_1\times(m_1+1)$-matrix and $\tilde A_2$ is an 
$m_2\times (m_2 + 1)$-matrix. 
In particular, the Newton polytopes $\N_{A_1}$ and
$\N_{A_2}$ has $\0$ as a relatively interior point, and as their only
intersection point.

With $A$ in the form \eqref{eqn:Aunionofsimplicesform}, 
we can act by integer affine
transformations affecting $\tilde A_1$ and $\tilde A_2$ separately.
Therefor, if $A$ is in orthagonal form, then we can assume that
\begin{equation}
\label{eqn:formoftildeA}
 \tilde A_k = (-p_1\e_1, \dots, -p_{m_k}\e_{m_k}, \a_{m_k+1}),
\end{equation}
where $p_1, \dots, p_{m_k}$ are positive integers, 
and hence $\a_{m_k+1}$ has only
positive coordinates.
We will say that $A$ is in \emph{special orthogonal form} if 
\eqref{eqn:formoftildeA} holds.
The main result of this section is the following lemma and theorem.

\begin{lemma}
\label{lem:orthogonalForm}
 Each circuit $A$ can be put in (special) orthogonal form by applying
 an integer affine transformation.
\end{lemma}

\begin{theorem}
\label{thm:criticalpoints}
 Let $A$ be a circuit in special orthogonal form.
 Then, for each $f\in \CC_*^A$, the set of critical arguments
 is an index set of the complement of $\overline{\cA}_f$.
\end{theorem}
 
The conditions of Theorem~\ref{thm:criticalpoints}
can be relaxed in small dimensions. When $n=1$, it is enough
to require that $\0$ is an interior point of $\N_A$.
When $n=2$, for generic $f$, 
it is enough to require that each quadrant $Q$ fulfills that
$\overline Q\setminus\{\0\}$ has nonempty intersection with $A$.

 \begin{proof}[Proof of Lemma \ref{lem:orthogonalForm}]
  Let $\u_1\dots, \u_{m_2}$ be a basis for the left kernel $\ker(A_1)$, 
  and let $\v_1, \dots, \v_{m_1}$ be a basis for the left kernel $\ker(A_2)$.
  Multiplying $A$ from the left by
   \[
   T = \big(\e_1, \v_1, \dots, \v_{m_1}, \u_1,\dots, \u_{m_2}\big)^t,
  \]
  it takes the desired form. We need only to show that
  $\det(T)\neq 0$.
  
  Notice that $\ker(A_1)\cap\ker(A_2) = 0$, since $A$ is assumed to be of 
  full dimension. Assume that there is a linear combination
  \[
   \lambda_0 \e_1 + \sum_{i=1}^{m_1} \lambda_i \v_i + \sum_{j=1}^{m_2} \lambda_j \u_j = 0.
  \]
  Then, since $B$ is a Gale dual of $A$,
  \[
   0 = \left(\sum_{j=1}^{m_2} \lambda_j \u_j\right) A B = (0,\dots, 0, -\lambda_0, \dots, -\lambda_0)B = -\lambda_0\sum_{\a\in A_2}\b_\a = \lambda_0\Vol(A),
  \]
  and hence $\lambda_0 = 0$. This implies that
  $\sum_{i=1}^{m_1} \lambda_i \v_i \in \ker(A_2)$, and hence
  $\sum_{i=1}^{m_1} \lambda_i \v_i = \0$. Thus, $\lambda_i = 0$ for all $i$ 
  by linear independence of the vectors $\v$. 
  Then, linear independence of the vectors $\u_j$ imply that
  $\lambda_j = 0$.
 \end{proof}

 \begin{proof}[Proof of Theorem~\ref{thm:criticalpoints}]
 We find that
 \begin{align*}
 z_i\partial_i f(z) &= -p_i f_iz^{\a_i} + \<\a_{m_k}, \e_i\> f_{m_1}z^{\a_{m_1}}, \quad i = 0, \dots, m_1\\
  z_j\partial_j f(z) &= -p_j f_iz^{\a_j} + \<\a_{n+1}, \e_j\> f_{n+1}z^{\a_{n+1}}, \quad j = m_1+1, \dots, n.
 \end{align*}
 Hence, for each $\theta\in \I$, it holds that
 \begin{equation}
 \label{eqn:ciriticalpointmonomialarguments}
 \hat f_{0}(\theta) = \dots = \hat f_{m_1}(\theta) 
 \quad \text{and} \quad
 \hat f_{m_1+1}(\theta) = \dots = \hat f_{n+1}(\theta).
 \end{equation}
 In particular, $f$ is colopsided at $\theta$ unless, after a rotation, $\hat f_{k}(\theta) = \delta_k$
 for all $k$. In the latter case, we refer to Theorems~\ref{thm:configurationspaceinteriorpoint} and ~\ref{thm:configurationspacevertex}. 

To see that the points $\theta\in\I$ for which $f$ is colopsided at $\theta$ are contained
in distinct connected components of the complement of $\overline{\cL}_f$, consider a line segment $\ell$
in $\RR^n$ with endpoints in $\I$. Then, not all identities of \eqref{eqn:ciriticalpointmonomialarguments}
can hold identically along $\ell$. Since the argument of each monomial is linear in $\theta$, this
implies that for a pair such that the identity in \eqref{eqn:ciriticalpointmonomialarguments}
does not hold identically along $\ell$, there is an intermediate point $\theta\in\ell$ for which the
corresponding monomials are antipodal, and hence $\theta\in\overline{\cL}_f$.
\end{proof}

\section{On systems supported on a circuit}
\label{sec:systems}
In this section we will consider a system 
\begin{equation}
\label{eqn:system1}
 F_1(z) = F_2(z) = 0
\end{equation}
of two bivariate polynomials. We will write $f(z) = 0$ for the
system \eqref{eqn:system1}. The system is said to be generic if it
has finitely many roots in $\CC_*^2$, and it is
said to be supported on a circuit $A$ if the supports
of $F_1$ and $F_2$ are contained in, but not necessarily equal to, $A$.
That is, we allow coefficients in $\CC$ rather than $\CC_*$.
By the Bernstein--Kushnirenko theorem, a generic system $f(z) = 0$ has 
at most $\Vol(A)$-many roots in $\CC_*^2$.
However, if $f$ is real, then fewnomial
theory states that a generic system $f(z)$ has at most three roots 
in $\RR_+^2 = \Arg^{-1}(0)$.
We will solve the complexified fewnomial problem,
i.e., for $f(z)$ with complex coefficients we will bound the number of 
roots in each sector $\Arg^{-1}(\theta)$.
Our intention is to offer a new approach to
fewnomial theory. We will restrict to
the case of simplex circuits, for the following two reasons.
Firstly, it allows for a simpler exposition. 
Secondly, for vertex circuits our method recovers
the known (sharp) bound, while for simplex circuits 
we obtain a sharpening of the fewnomial bound.

\begin{theorem}
\label{thm:FewnomialSimplexCircuits}
 Let $f(z) = 0$ be a generic system of two bivariate polynomials
 supported on a planar simplex circuit $A \subset \ZZ^2$.
 Then, each sector $\Arg^{-1}(\theta)$ 
 contains at most two solutions of $f(z) = 0$.
\end{theorem}

\subsection{Reducing $f(z)$ to a system of trinomials}
A generic system $f(z)$ is, by taking appropriate linear combinations,
equivalent to a system of two trinomials whose support intersect in a dupleton.
That is, we can assume that $f(z)$ is in the form
\begin{equation}
\label{eqn:circuitsystem}
 \left\{\begin{array}{lllllll}
         F_1(z) = &f_1z^{\a_0} & + &  z^{\a_2} & + & f_2 z^{\a_3}& = 0\\
         F_2(z) = & f_3 z^{\a_1} & + & z^{\a_2} & + & f_4 z^{\a_3} & = 0,
        \end{array}\right.
\end{equation}
with coefficients in $\CC_*$. We will use the notation
\[
 A = \left(\begin{array}{cccc}
            1 & 1 & 1 & 1 \\
            \a_0 & \a_1 & \a_2 & \a_3
           \end{array}\right)
\quad \text{and} \quad
 \hat A = \left(\begin{array}{cc}
            1 & 0 \\
            0 & 1 \\
            A_1 & A_2
           \end{array}\right),
\]
where $A_k$ denotes the support of $F_k$ (notice that this 
differs from the notation used in previous sections).
Notice that we can identify a system $f(z)$ in the form
\eqref{eqn:circuitsystem} with its corresponding vector 
in $\CC_*^{\hat A}$.

When reducing $f(z)$ to the form \eqref{eqn:circuitsystem}
by taking linear combinations,
there is a choice of which monomials to eliminate in $F_1$ and
$F_2$ respectively. In order for the arguments of the roots of
$f(z)=0$ to depend continously on the coefficients,
we need to be careful with which choice to make.

\begin{lemma}
\label{lem:argumentsvarycontinuously}
Let $\ell$ denote the line through $\a_2$ and $\a_3$, and let $\gamma$ be a 
compact path in $\CC_*^{\hat A}$. If $\ell$ intersect the interior of $\N_A$,
then the arguments of the solutions to $f(z) = 0$ vary continuously along 
$\gamma$.
\end{lemma}

\begin{proof}
It is enough to show that along a compact path $\gamma$, the set 
\begin{equation}
\label{eqn:amoebaovergamma}
\bigcup_{f\in \gamma}\Am_f = \bigcup_{f\in\gamma}\Log(Z(f))
\end{equation}
is bounded, for it implies that for $f\in\gamma$, the roots of
$f$ are uniformly separated from the boundary of $\CC_*^{\hat A}$.

We first claim that our assumptions imply that the
normal fans of $\N_{A_1}$ and $\N_{A_2}$ has no coinciding one dimensional 
cones.
Indeed, these fans has a coinciding one dimensional cone if and only if the Newton
polytopes $\N_{A_1}$ and $\N_{A_2}$ has facets $\Gamma_1$ and $\Gamma_2$
with a common outward normal vector $\n$.
As $A$ is a circuit, it holds that $\Gamma_1 = \Gamma_2 = [\a_2, \a_3] \subset \ell$.
Since the normal vector $\n$ is common for $\N_{A_1}$ and $\N_{A_2}$, we find that
$\Gamma_1$ (and $\Gamma_2$) is a facet of $\N_A$. But then $\ell$ contains a
facet of $\N_A$, and hence it cannot intersect the interior of $\N_A$, a contradiction.

Consider a point $f\in \CC_*^A$. Since the normal fans $\N_{A_1}$ and $\N_{A_2}$
has no coinciding one dimensional cones, the intersection of the amoebas $\Am_{F_1}$ and $\Am_{F_2}$
is bounded (this follows, e.g., from the fact the amoeba has finite 
Hausdorff distance from the Archimedean tropical variety, 
see \cite{AKNR13}). Thus, the amoeba $\Am_f$ is bounded, say that $\Am_f\subset D(R_f)$ where $D(R_f)$
denotes the disk of radii $R_f$ centered around the origin. 
By continuity of roots, $\Am_{\tilde f}\subset D(R_f)$
for all $\tilde f$ in some neighborhood $N_f$ of $f$.
The compactness of $\gamma$ implies our result.
\end{proof}

In order for the assumptions of Lemma~\ref{lem:argumentsvarycontinuously}
to be fulfilled, for a simplex circuit $A$, we need that $\a_0$ and $\a_1$ are vertices
of $\N(A)$, see Figure~\ref{fig:systeminteriorpolytopes}.

\begin{figure}[h]
\centering
\includegraphics[height=20mm]{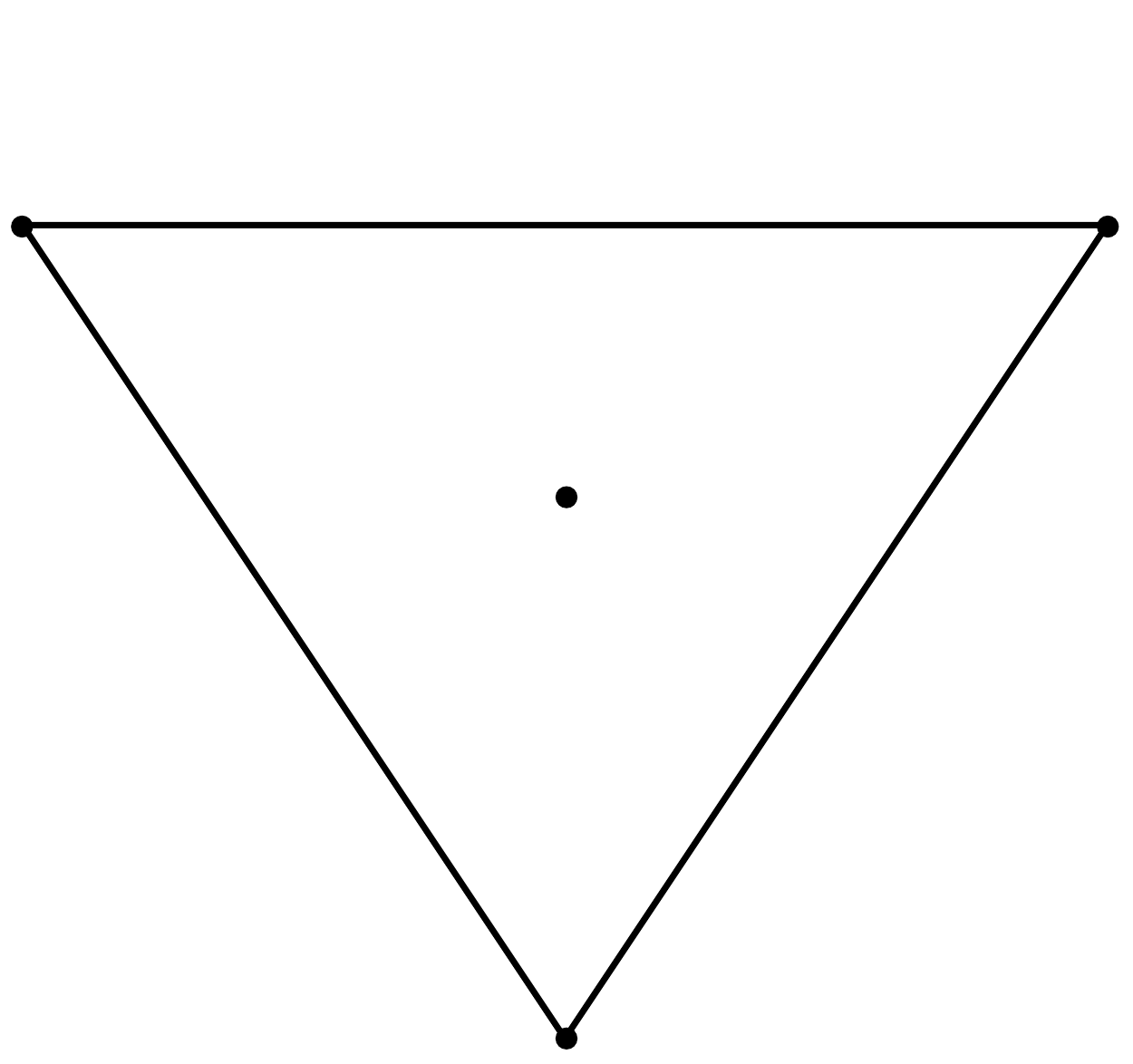}
\hspace{10mm}
\includegraphics[height=20mm]{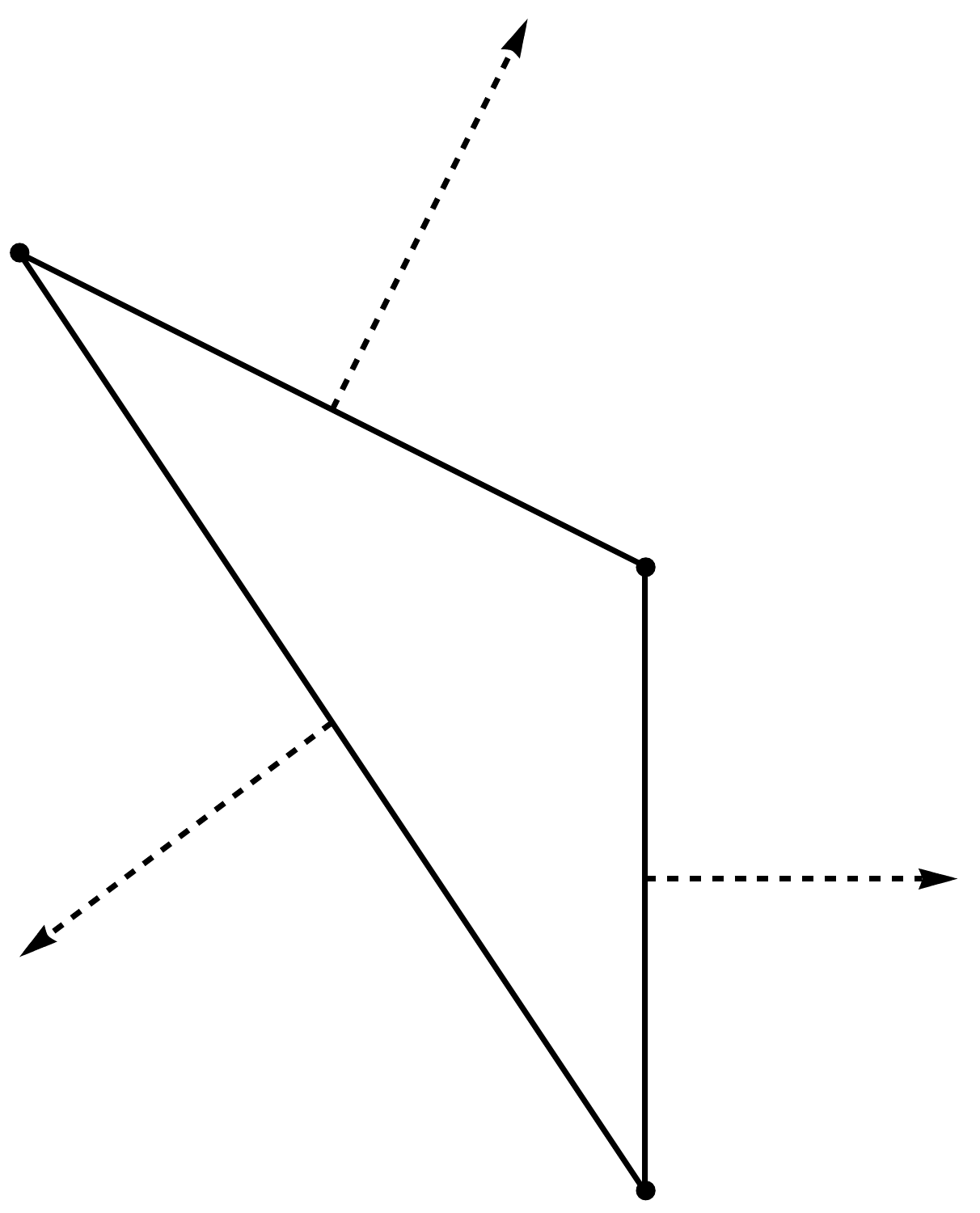}
\hspace{10mm}
\includegraphics[height=20mm]{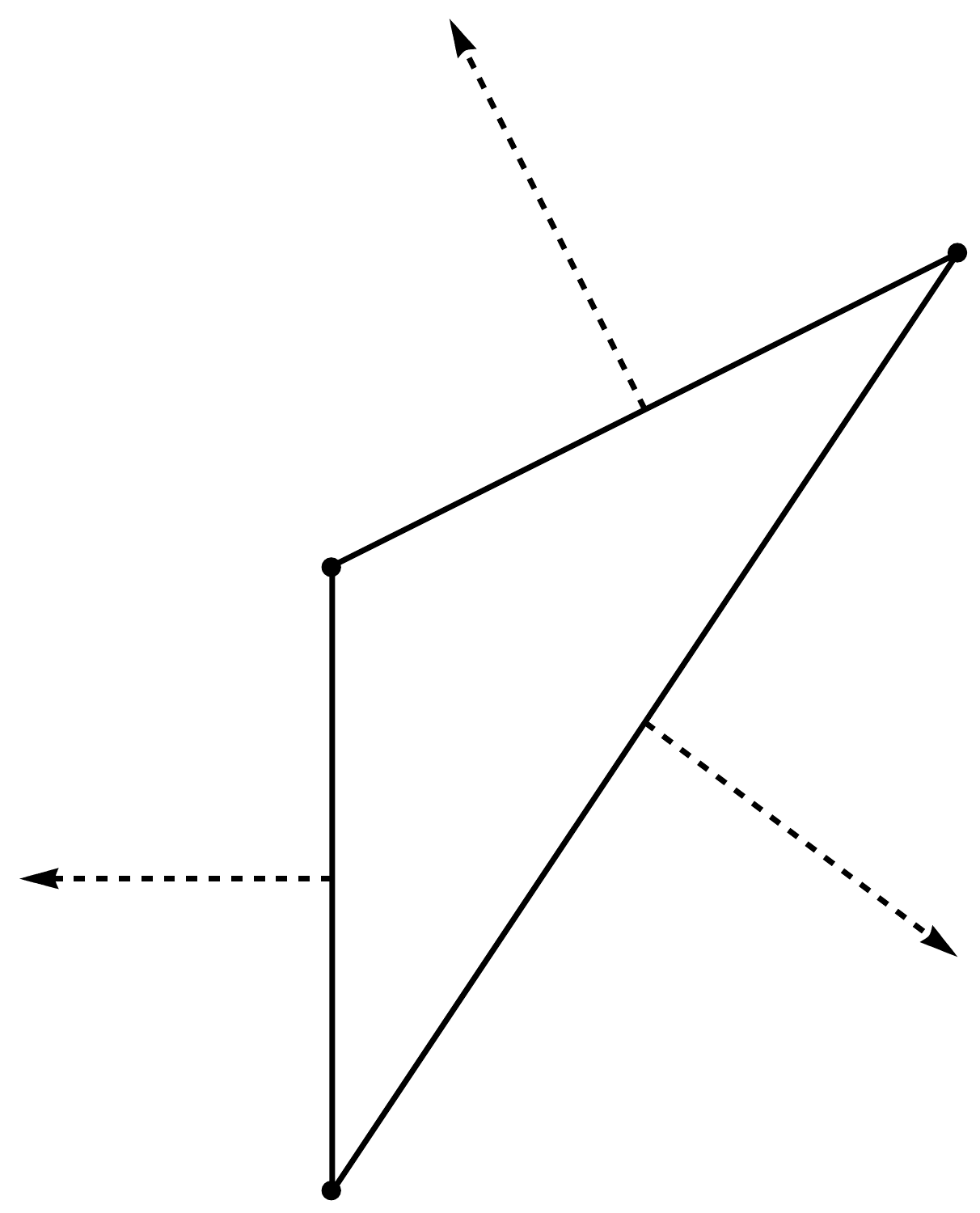}
\caption{The Newton polytopes $\N_A$, $\N_{A_1}$, and $\N_{A_2}$.}
\label{fig:systeminteriorpolytopes}
\end{figure}

\begin{proposition}
\label{pro:nonrealsystem}
If $f$ is nonreal at $\theta$, then
there is at most one zero of $f(z)=0$ contained in the sector 
$\Arg^{-1}(\theta)$.
\end{proposition}

\begin{proof}
If $F_k$ is nonreal, then the fiber in $Z(F_k)$ over a point 
$\theta \in \cA_{F_k}$
is a singleton. Hence, if the number of roots of $f(z) = 0$ in
$\Arg^{-1}(\theta)$ is greater than one, then both $F_1$ and $F_2$ are
real at $\theta$.
\end{proof}

The implication of Proposition~\ref{pro:nonrealsystem} is that 
the complexified fewnomial problem reduces to the real
fewnomial problem. However, our approach is dependent on
allowing coefficients to be nonreal.
In fact, we will consider a partially complexified problem,
allowing $f_1, f_3\in \CC_*$ but requiring $f_2, f_4 \in \RR_*$.

\subsection{Colopsidedness}
We define the colopsided coamoeba of the system $f(z)$ by
\[
\cL_f = \cL_{F_1}\cap \cL_{F_2} = \cA_{F_1}\cap\cA_{F_2},
\]
where the last equality follows from \cite[cor.\ 3.3]{FJ15}. 
That is, $f$ is said to be colopsided at $\theta$ if either $F_1$ or $F_2$ is
colopsided at $\theta$.
We will say that $f$ is real at $\theta$ if both $F_1$ and $F_2$ are real at $\theta$.

The lopsided coamoeba $\cL_f$ consist of a number of polygons on $\T^2$,
possibly degenerated to singletons. The following two lemmas
will allow us to count the number of such polygons.

\begin{lemma}
\label{lem:binomialcoamoebaincomplement}
Assume that $f$ is nonreal.
Let $g$ be a binomial constructed by choosing two monomials from \eqref{eqn:circuitsystem},
possibly alternating signs.
 If $f_2$ and $f_4$ are of opposite signs, 
 then $\cA_g \subset \T^2\setminus\cL_f$.
 If $f_2$ and $f_4$ are of equal signs, then 
 $\cA_g \subset \T^2\setminus\cL_f$ \emph{except} for
 $g(z) = \pm(f_1z^{\a_0} - f_3z^{\a_1})$.
\end{lemma}
\begin{proof}
 If, for $\theta\in \T^2$, two components of $\hat F_1(\theta)$
 is contained in a real subvector space
 $\ell\subset \CC$, then either $F_1$ is colopsided at $\theta$
 or $\hat F_1(\theta) \subset \ell$. However, the latter implies that
 two components of $\hat F_2(\theta)$ are contained in $\ell$. 
 Repeating the argument yields that either $f$ is real, or it is colopsided at $\theta$.
 
 Thus, the only binomials we need to consider is 
 $g_\pm(z) = f_1z^{\a_0} \pm f_3z^{\a_1}$.
 For each $\theta \in \cA_{g_+}$ the vectors 
 $\hat F_1(\theta)$ and $\hat F_2(\theta)$ differ in sign in their
 first component, and hence at least one is colopsided at 
 $\theta$, unless $f$ is real. 
 For each $\theta \in \cA_{g_-}$, the vectors
 $\hat F_1(\theta)$ and $\hat F_2(\theta)$ differ in signs 
 in the the last component only if $f_2$ and $f_4$
 differ in signs. If this is the case, then at least one is 
 colopsided at $\theta$ unless $f$ is real.
\end{proof}

\begin{lemma}
\label{lem:shellpointsofsystemiscontainedinclosure}
 Let $\theta\in \cA_{g_1}\cap\cA_{g_2}$
 for truncated binomials $g_1$ and $g_2$ of $F_1$ and $F_2$ respectively.
 If the Newton polytopes (i.e., line segments) of $g_1$ and $g_2$ are 
 nonparallel, then $\theta\in \overline{\cL}_f$.
\end{lemma}

\begin{proof}
 If $F_1$ and $F_2$ are both real at $\theta$, then $\theta \in \cL_f$.
 If $F_1$ is nonreal at $\theta$, then for a sufficiently small neighborhood
 $N_\theta\subset\RR^2$, it holds that
 \[
  \cA_{F_1}\cap N_\theta = \{\varphi\,|\, \<\varphi,\n\> > \<\theta,\n\>\}\cap N_\theta,
 \]
 where $\n$ is a normal vector of $\N(g_1)$. Since connected components of the
 complement of $\cA_{F_2}$ are convex,
 either $\cA_{F_2}$ intersect $\cA_{F_1}$ in $N_\theta$, or the boundary of
 $\cA_{F_2}$ is contained in the line $\ell = \{\varphi\,|\, \<\varphi,\n\> =\<\theta,\n\>\}$.
 As the boundary of $\cA_{F_2}$ contains $\cA_{g_2}$, it holds in the latter case
 that $\cA_{g_2}\subset\ell$, which in turn implies that $\n$ is a normal
 vector of $\N(g_2)$, contradicting our assumptions. 
 We conclude that $\cA_{F_2}\cap\cA_{F_1}\cap N_\theta \neq \emptyset$.
 Since this holds for any sufficiently small neighborhood $N_\theta$, 
 the result follows.
\end{proof}

\begin{example}
\label{ex:LopsidedSystem}
Consider the system
\[
  f(z) = \left\{\begin{array}{llllll}
          f_1z_1z_2^2 +  1 +  f_2 z_1z_2\\
          f_3 z_1^2 z_2 + 1 +  f_4 z_1z_2.
        \end{array}\right.
\]
We have that $\Vol(A) = 3$. Hence $H$ divides $\T^2$
into three cells. The lopsided coamoeba $\cL_f$, and
the hyperplane arrangement $H$, can be seen in
Figure~\ref{fig:systeminteriorexplained2}.
In the first two picture, the generic respectively real 
situation when $f_2$ and $f_4$ differs in signs. 
In last two pictures, the generic respectively real
situation when $f_2$ and $f_4$ have equal signs.
In the generic case, the lopsided coamoeba $\cL_f$ consist
of three polygons. When deforming from the generic to
the real case, we observe the following behavior.
Some polygons of $\cL_f$ deform into single points - by necessity
points contained in the lattice $P$.
Some pairs of polytopes of $\cL_f$ deforms to nonconvex polygons,
typically with a single intersection point. 
Our proof of Theorem~\ref{thm:FewnomialSimplexCircuits}
is based on the observation that, when deforming from a generic to a real system,
at most two polytopes of $\cL_F$ deforms a nonconvex polygon 
intersecting $H$. 

 \begin{figure}[h]
\centering
\includegraphics[height=26mm]{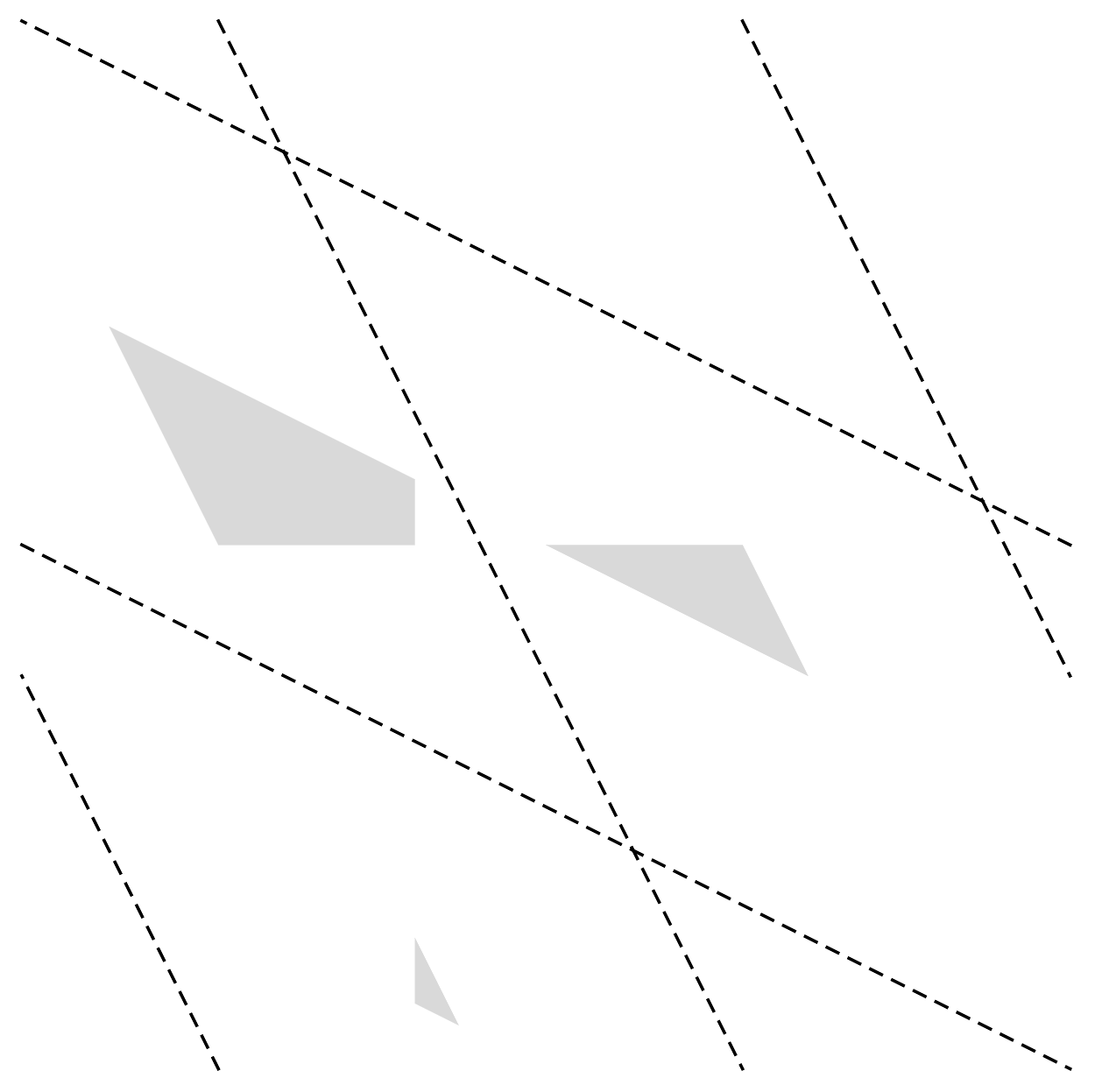}
\hspace{4mm}
\includegraphics[height=26mm]{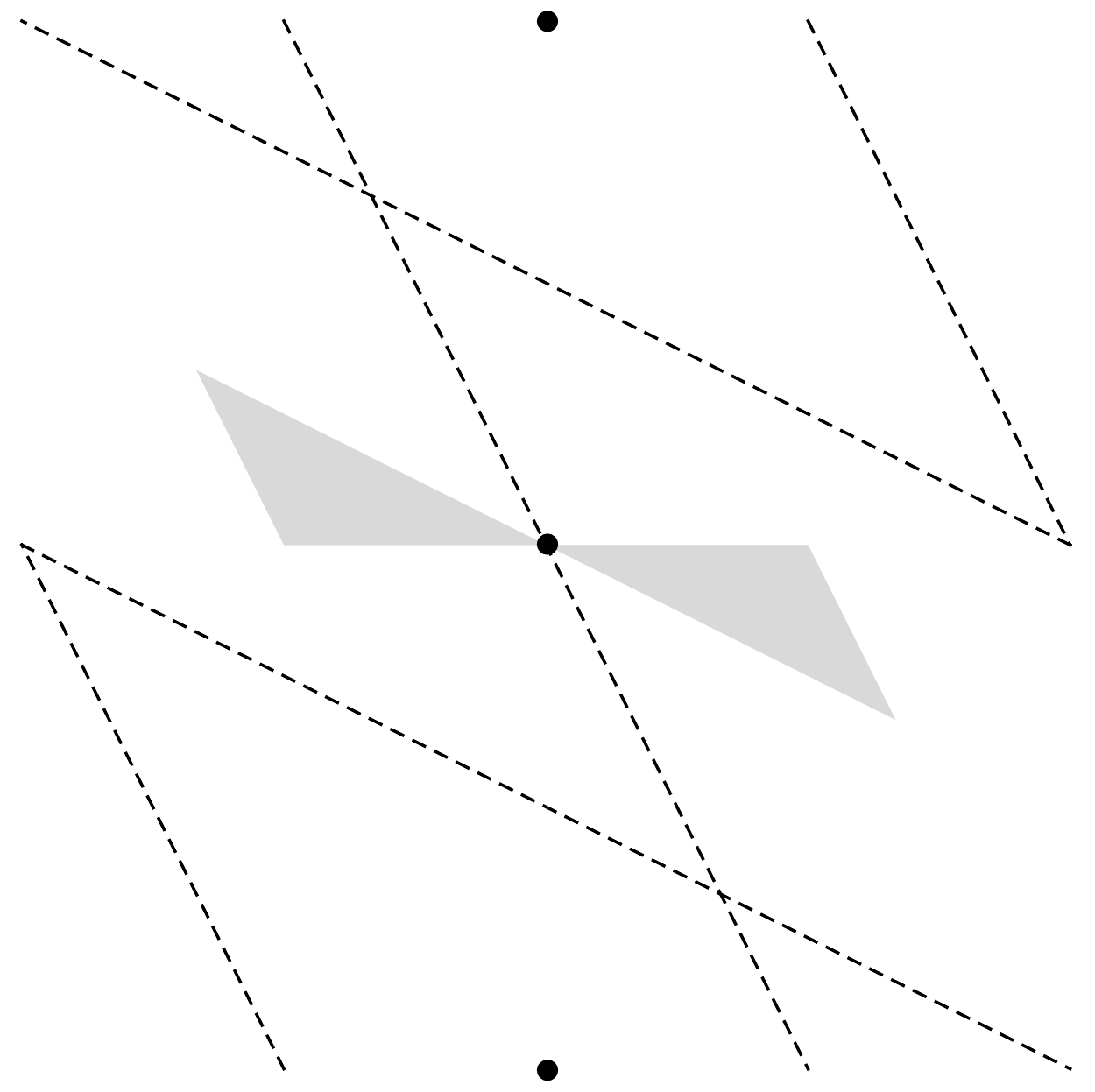}
\hspace{4mm}
\includegraphics[height=26mm]{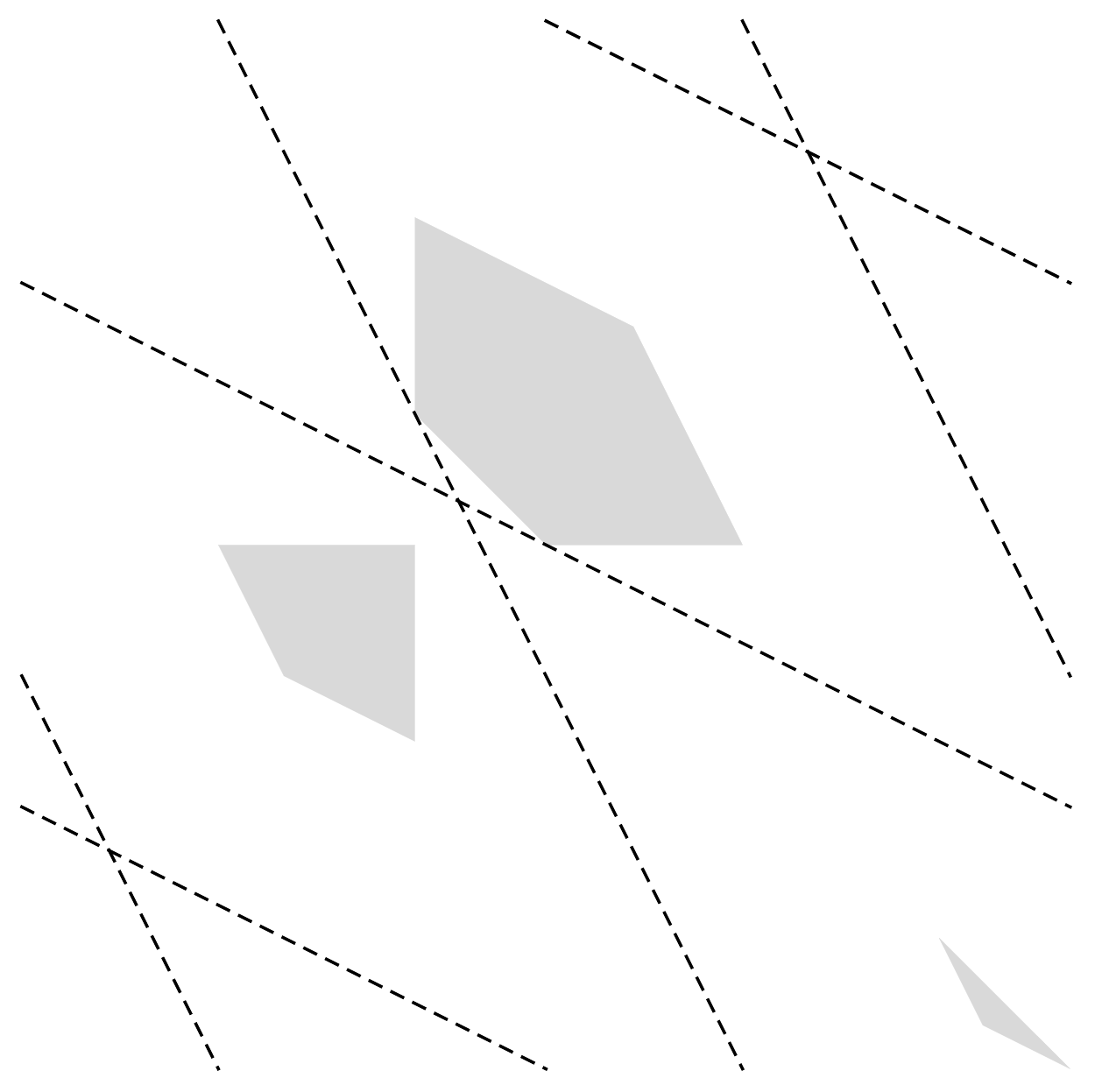}
\hspace{4mm}
\includegraphics[height=26mm]{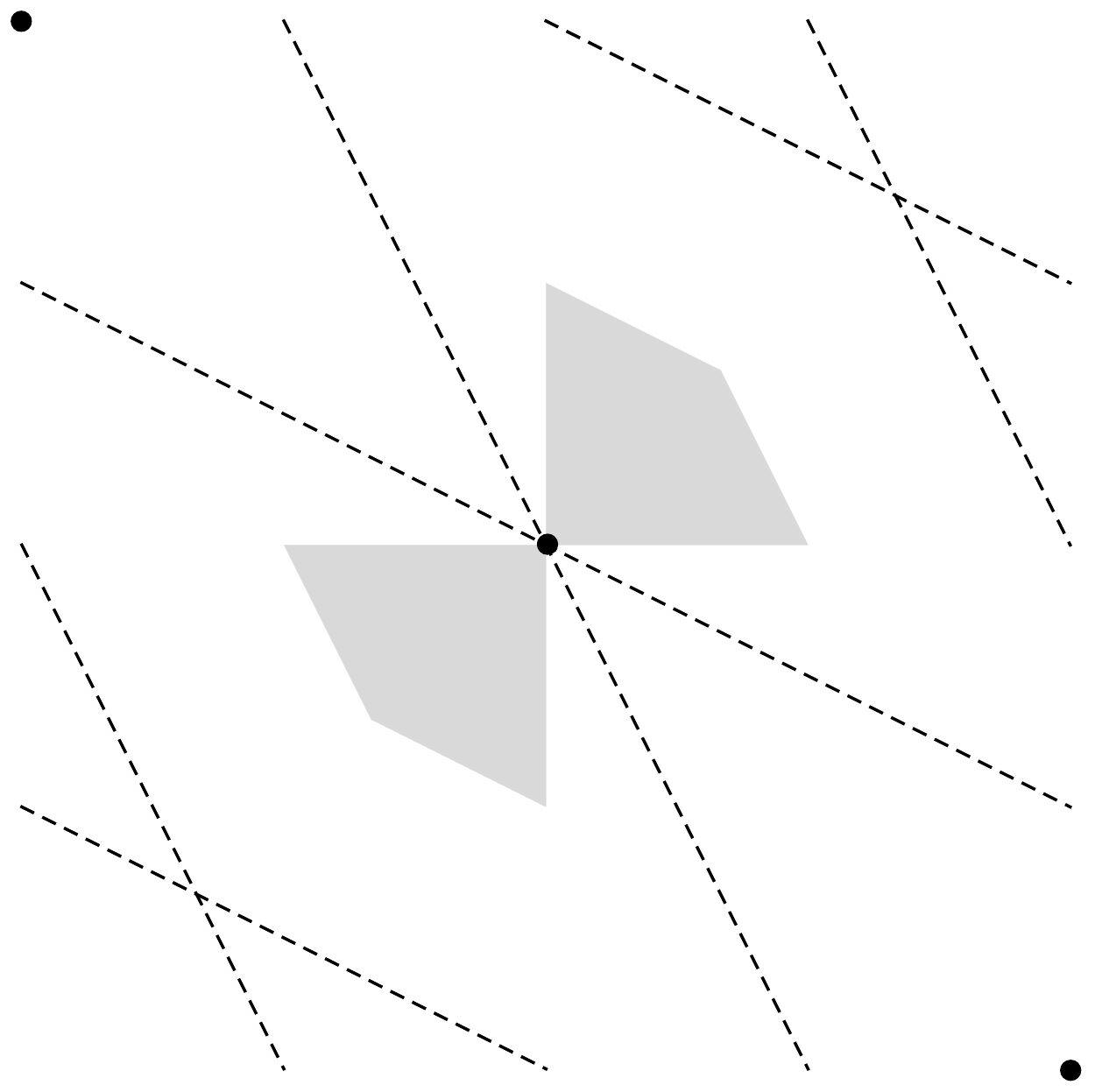}
\caption{The lopsided coamoebas from Example~\ref{ex:LopsidedSystem}.}
\label{fig:systeminteriorexplained2}
\end{figure}
\end{example}

\subsection{Proof of Theorem \ref{thm:FewnomialSimplexCircuits}}
Let us consider the auxiliary binomials
\[
\begin{array}{lll}
g_1(z) = f_1z^{\a_0} - z^{\a_2},&  & g_2(z) = f_3z^{\a_1} - z^{\a_2},\\
h_1(z) = f_1z^{\a_0} + z^{\a_2}, & \quad \text{and} \quad & h_2(z) = f_3z^{\a_1} + z^{\a_2}.
\end{array}
\]
The vectors $\a_2 - \a_0$ and $\a_2-\a_1$ span the simplex $\N_A$, 
hence the hyperplane arrangement 
$H = \cA_{g_1}\cup\cA_{g_2}$ 
divides $\T^2$ into $\Vol(A)$-many
parallelograms with the points $P = \cA_{h_1}\cap \cA_{h_2}$ as 
their centers of mass.

If $f$ is nonreal, then Lemma~\ref{lem:binomialcoamoebaincomplement}
shows that $H \subset \T^2\setminus\cL_f$, and 
Lemma~\ref{lem:shellpointsofsystemiscontainedinclosure}
shows that $P \subset \overline{\cL}_f$. 
By Lemma~\ref{lem:argumentsvarycontinuously} we find that
$\cL_f$ has at most $\Vol(A)$-many connected components.
Hence, $\cL_f$ has at exactly one connected component in each of the cells of $H$,
and the number of roots of $f(z) = 0$ projected by the argument map into each such
component is exactly one.

Consider now the real case when $f_2$ and $f_4$ differs in signs.
Then, at least one of $F_1$ and $F_2$ are colopsided at the 
intersection points $\cA_{g_1}\cap\cA_{g_2}$.
Thus, if $\Arg^{-1}(\theta)$ contains a root of $f(z) = 0$, then 
a sufficiently small neighborhood $N_\theta$ intersect at most two
of the cells of the hyperplane arrangement $H$. 
Hence, using Lemma~\ref{lem:argumentsvarycontinuously} and
wiggling the arguments of coefficients of $f$ by $\varepsilon$, 
$N_\theta$ intersect at most two of the polygons
of $\cL_{f^\varepsilon}$. Hence, 
there can be at most two roots contained in $\Arg^{-1}(\theta)$.

Consider now the case when $f$ real with $f_2$ and $f_4$ of equal signs. 
In this case, a point $\theta\in \cA_{g_1}\cap\cA_{g_2}$ can be contained in
$\cL_f$. See the left picture of Figure~\ref{fig:systeminteriorproof}, 
where the hyperplane arrangement $H$
is given in black, and the shells $\H_{F_1}$ and $\H_{F_2}$ are given 
in red and blue respectively, with indicated orientation.
Wiggling the arguments of coefficient $f_1$ and/or $f_3$ by $\varepsilon$, 
we claim the we obtain a situation as in the right picture
of Figure~\ref{fig:systeminteriorproof}. That is, at most two polygons 
of $\cL_{f^\varepsilon}$ will intersect a small neighborhood $N_\theta$ of $\theta$.
Let us prove this last claim.

Let $f$ be generic, with $f_2$ and $f_4$ real and of equal signs.
The hyperplanes $\cA_{g_1}$ and $\cA_{g_2}$ (locally) divides the plane into
four regions.  We can assume that $\a_2 = \0$. Then, $\cA_{g_1}$ consist of
all $\theta$ such that $\hat f_1(\theta) = 1$, and $\cA_{g_1}$ consist of
all $\theta$ such that $\hat f_3(\theta) = 1$. Thus, locally, the cells of
$H$ can be indexed by the signs of the imaginary parts of $\hat f_1(\theta)$ and $\hat f_3(\theta)$.
Assume that $\tilde \theta\in \cL_f\cap N_\theta$. Then neither $F_1$ nor $F_2$ is colopsided
at $\tilde \theta$. Observe that $\hat f_2(\tilde \theta) = \hat f_4(\tilde \theta)$, since
$f_2$ and $f_4$ has equal sign. 
We find that
\[
\sgn(\Im(\hat  f_1(\tilde \theta))) = - \sgn(\Im(\hat  f_2(\tilde \theta))) = - \sgn(\Im(\hat  f_4(\tilde \theta))) =  \sgn(\Im(\hat  f_4(\tilde \theta))),
\]
where the first and the last equality holds since neither $F_1$ nor $F_2$ is 
colopsided at $\tilde \theta$. 
This implies that polygons of $\cL_f$ intersecting a 
small neighbourhood of $\theta$ are necessarily contained in the
cells of $H$ which corresponds to that the imaginary parts of 
$\hat f_1(\tilde \theta)$ and $\hat f_4(\tilde \theta)$
have equal signs. As there are two such cells, we find that there are at most
two polygons of $\cL_f$ intersecting a small neighbourhood of $\theta$.

\begin{figure}[h]
\centering
\includegraphics[height=30mm]{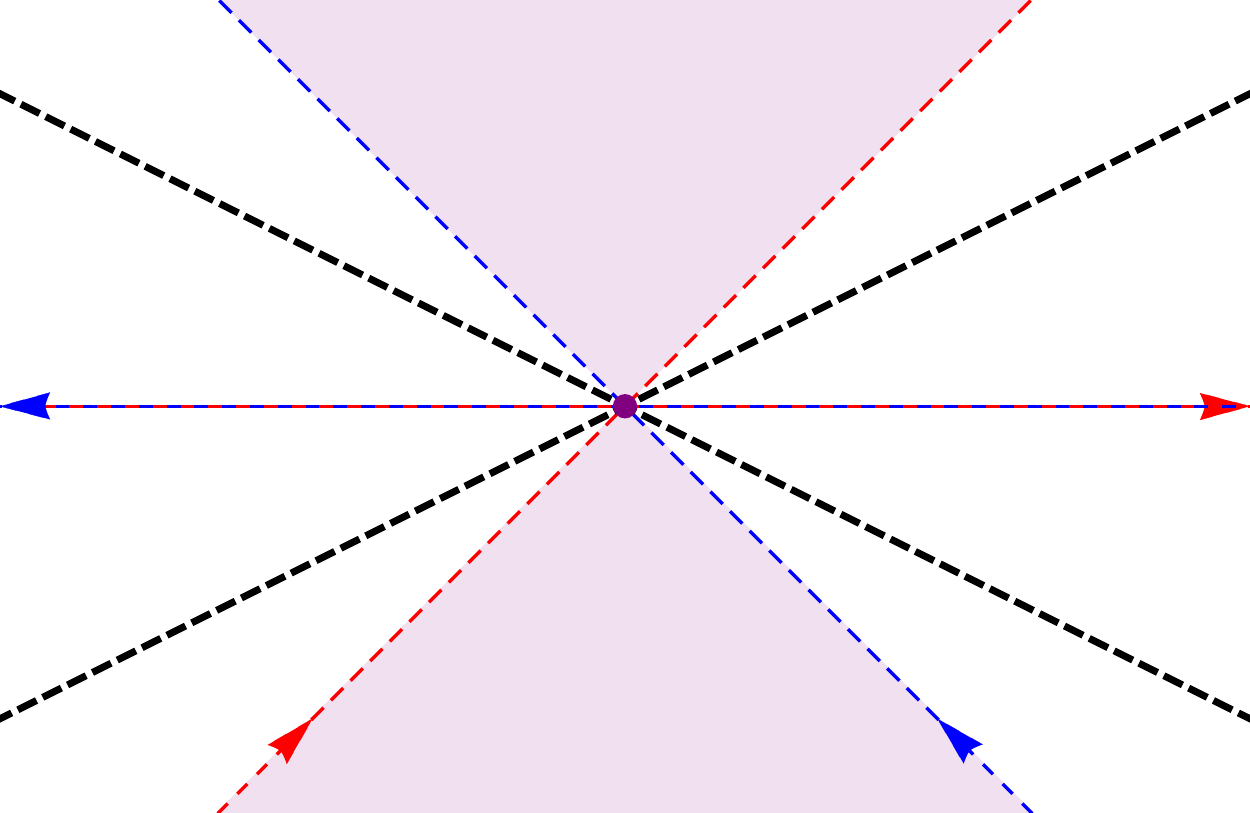}
\hspace{10mm}
\includegraphics[height=30mm]{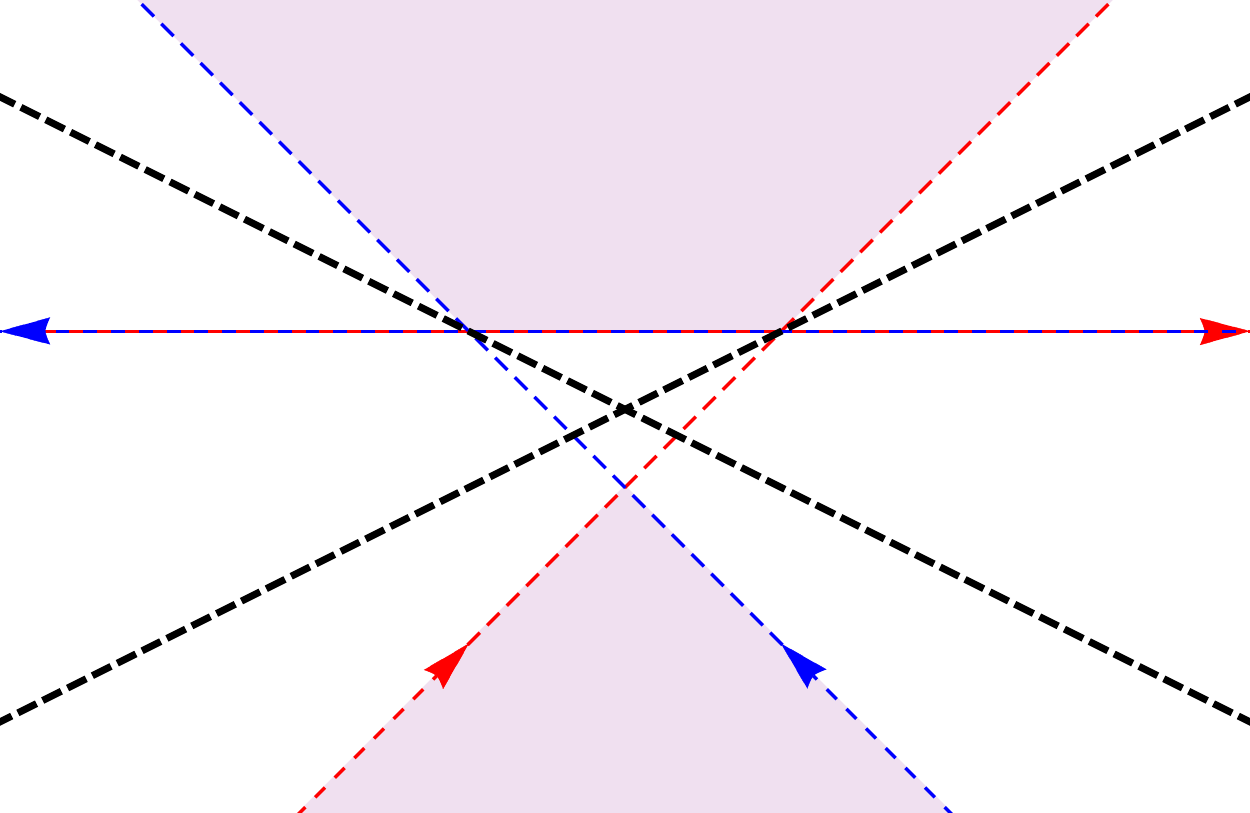}
\caption{To the left: the coamoeba $\cL_f$ close to a point of
$\cA_{g_1}\cap\cA_{g_2}$
when $f_2$ and $f_4$ have equal in signs and $f_1$ and $f_3$ are real.
To the right: the same picture after wiggling the argument of $f_1$ or $f_3$.}
\label{fig:systeminteriorproof}
\end{figure}


\end{document}